\documentclass[11pt,reqno,a4paper]{amsart}

\usepackage{fbb}
\usepackage[T1]{fontenc}  
\usepackage[utf8]{inputenc}
\usepackage{lmodern}
\usepackage{microtype}
\usepackage{marginnote}

\usepackage{amsmath,amssymb,amsfonts,amsthm,mathtools,mathrsfs}
\usepackage{xcolor}
\usepackage{url}
\usepackage[noadjust]{cite}

\usepackage[top=2.5cm,bottom=2.4cm,left=2.8cm,right=2.8cm,headsep=0.2in,headheight=14pt]{geometry}
\usepackage{fancyhdr}
\usepackage[colorlinks=true,
  linkcolor=blue!50!green,
  citecolor=blue!50!green,
  urlcolor=blue!50!green]{hyperref}
\hypersetup{
  pdfauthor={Mayukh Mukherjee, Soumyadeb Samanta and Soumyadip Thandar},
  pdftitle={Pointed blow-downs, cohomology, and multiplicative gaps for nilpotent quasi-isometries},
  pdfsubject={Quasi-isometries of finitely generated nilpotent groups},
  pdfkeywords={nilpotent groups, quasi-isometries, asymptotic cones, Pansu differentials, invariant means, cohomological induction}
}
\usepackage[nameinlink,capitalize,noabbrev]{cleveref}

\theoremstyle{plain}
\newtheorem{theorem}{Theorem}[section]
\crefname{theorem}{theorem}{theorems}
\Crefname{theorem}{Theorem}{Theorems}
\newtheorem{lemma}[theorem]{Lemma}
\crefname{lemma}{lemma}{lemmas}
\Crefname{lemma}{Lemma}{Lemmas}
\newtheorem{proposition}[theorem]{Proposition}
\crefname{proposition}{proposition}{propositions}
\Crefname{proposition}{Proposition}{Propositions}
\newtheorem{corollary}[theorem]{Corollary}
\crefname{corollary}{corollary}{corollaries}
\Crefname{corollary}{Corollary}{Corollaries}

\theoremstyle{definition}
\newtheorem{definition}[theorem]{Definition}
\crefname{definition}{definition}{definitions}
\Crefname{definition}{Definition}{Definitions}

\crefname{example}{example}{examples}
\Crefname{example}{Example}{Examples}

\theoremstyle{remark}
\newtheorem{remark}[theorem]{Remark}
\crefname{remark}{remark}{remarks}
\Crefname{remark}{Remark}{Remarks}

\theoremstyle{plain}
\newtheorem{mainthm}{Theorem}

\crefname{mainthm}{Theorem}{Theorems}
\Crefname{mainthm}{Theorem}{Theorems}

\crefname{section}{section}{sections}
\Crefname{section}{Section}{Sections}
\crefname{appendix}{Appendix}{Appendices}
\Crefname{appendix}{Appendix}{Appendices}
\crefname{subsection}{subsection}{subsections}
\Crefname{subsection}{Subsection}{Subsections}
\crefname{subsubsection}{subsection}{subsections}
\Crefname{subsubsection}{Subsection}{Subsections}
\crefname{equation}{equation}{equations}
\Crefname{equation}{Equation}{Equations}

\newcommand{\RR}{\mathbb R}
\newcommand{\ZZ}{\mathbb Z}
\newcommand{\NN}{\mathbb N}
\newcommand{\Hom}{\operatorname{Hom}}

\begin{document}

\title[Cohomology and multiplicative gaps]{Cohomology and multiplicative gaps for quasi-isometries of nilpotent groups} \author{Mayukh Mukherjee}

\address{Department of Mathematics, Indian Institute of Technology Bombay, Powai, Mumbai 400076, India} \email{mathmukherjee@gmail.com}

\author{Soumyadeb Samanta} \address{Department of Mathematics, Indian Institute of Technology Bombay, Powai, Mumbai 400076, India} \email{soumyadeb@iitb.ac.in}

\author{Soumyadip Thandar} \address{Theoretical Statistics and Mathematics Unit, Indian Statistical Institute, Kolkata 700108, India} \email{stsoumyadip@gmail.com}


\begin{abstract}
We construct two word metrics on the same seven-dimensional torsion-free nilpotent lattice whose asymptotic cones are isometric, but for which every quasi-isometry has multiplicative constant at least $1+\varepsilon$, for one fixed $\varepsilon>0$.

We use a convex invariant of a fixed quasi-isometry $\Psi:N\to M$. The displacement cocycle of $\Psi$ and the right-invariant means on $N$ determine a compact convex set $\mathcal A(\Psi)\subseteq\operatorname{Hom}(V_N,V_M)$. If $\mathscr G_{\mathrm{pt}}(\Psi)$ denotes the first-layer maps of the graded homomorphisms that occur as blow-downs of $\Psi$ with varying basepoints, then
$$
\mathcal A(\Psi)=\operatorname{conv}\mathscr G_{\mathrm{pt}}(\Psi), \qquad\operatorname{Ext}\mathcal A(\Psi)\subseteq\mathscr G_{\mathrm{pt}}(\Psi).
$$
For every extreme point $L\in\mathcal A(\Psi)$, the dual map $L^*$ is also the degree-one part of an isomorphism of the full real cohomology rings. This cohomological compatibility, together with the realization of first-layer Pansu differentials by invariant means, yields the multiplicative constant gap stated above.  
\end{abstract}

\maketitle

\section{Introduction}

The asymptotic cone records the first-order metric geometry of a nilpotent group. Even for two word metrics on the same group, one may ask whether an isometry of their cones can be reflected by quasi-isometries whose multiplicative constants approach one. We show that this need not happen. There is a seven-dimensional torsion-free nilpotent group $\Gamma$ with finite symmetric generating sets $S,T$ and $\varepsilon>0$ such that the asymptotic cones of $(\Gamma,d_S)$ and $(\Gamma,d_T)$ are isometric, but every $(\Lambda,A)$-quasi-isometry between them satisfies
$$
 \Lambda\ge 1+\varepsilon.
$$
The obstruction comes from the compatibility between averaged displacement maps of a quasi-isometry and the cohomology ring of $\Gamma$.

We first describe the object that detects this obstruction. Let $\Psi:N\to M$ be a quasi-isometry between finitely generated torsion-free nilpotent groups. Normalize $\Psi(e_N)=e_M$ and write
$$
 \Phi_n(x):=\pi_M\bigl(\Psi(x)^{-1}\Psi(xn)\bigr),
 \qquad x,n\in N,
$$
where $\pi_N:N\to V_N=N_{\mathrm{ab}}\otimes_\ZZ\RR$ and $\pi_M:M\to V_M=M_{\mathrm{ab}}\otimes_\ZZ\RR$ are the canonical maps. If $\mathsf A$ is a right-invariant mean on $N$, the identity 
$$
 L_{\mathsf A}\pi_N(n):=\mathsf A[\Phi_n]
$$
extends to a linear map $V_N\to V_M$. The resulting compact convex set is
$$
 \mathcal A(\Psi):=
 \{L_{\mathsf A}:\mathsf A\text{ is a right-invariant mean on }N\}.
$$
A fixed-basepoint blow-down of $\Psi$ is a bi-Lipschitz map $\Psi_\infty$ between the Carnot asymptotic cones $N_\infty$ and $M_\infty$, and is Pansu differentiable almost everywhere. Let $\mathscr D_1(\Psi)$ be the set of the first-layer differentials $D_p^1\Psi_\infty$ obtained from all such blow-downs and all their differentiability points. Our first result realizes each of these differentials by averages of $\Phi_n$.

\begin{mainthm}\label{thm:main-attainment}
Let $\Psi:N\to M$ be a normalized $(\Lambda,A)$-quasi-isometry between finitely generated torsion-free nilpotent groups with fixed word metrics. Suppose that the rescaled maps $\Psi_j:=\sigma_{r_j}^M\circ\Psi\circ(\sigma_{r_j}^N)^{-1}$ converge to $\Psi_\infty$ along $r_j\to\infty$, and let $p$ be a Pansu differentiability point of $\Psi_\infty$. There are centers $c_j\in N$ and integers $\rho_j\to\infty$ such that
$$
 \frac{\rho_j}{r_j}\longrightarrow0,
 \qquad \sigma_{r_j}^N(c_j)\longrightarrow p,
$$
and, simultaneously for every $n\in N$,
\begin{equation}\label{eq:intro-concrete-realization}
 \frac1{|B_{\rho_j}|}\sum_{x\in c_jB_{\rho_j}}\Phi_n(x)
 \longrightarrow D_p^1\Psi_\infty\,\pi_N(n).
\end{equation}
The translated balls $c_jB_{\rho_j}$ form a right F{\o}lner sequence. Moreover, $\mathcal A(\Psi)$ is compact and convex, is unchanged when $\Psi$ is replaced by a map at bounded distance from it, and
\begin{equation}\label{eq:intro-sandwich}
 \overline{\operatorname{conv}}\mathscr D_1(\Psi)
 \subseteq\mathcal A(\Psi)
 \subseteq
 \bigl\{L:\ell_{S_M}(Lv)\le\Lambda\ell_{S_N}(v)
                  \text{ for every }v\in V_N\bigr\}.
\end{equation}
If $\mathcal A(\Psi)$ is a singleton, then every fixed-basepoint blow-down of $\Psi$ is the same graded $\Lambda$-bi-Lipschitz isomorphism $N_\infty\to M_\infty$.
\end{mainthm}

Here $\rho_j\to\infty$ makes $(c_jB_{\rho_j})$ a F{\o}lner sequence, while $\rho_j/r_j\to0$ makes them collapse to $p$ in the cone. 

Allowing the basepoint to vary gives the following description of $\mathcal A(\Psi)$. For $c\in N$, set
$$
 \Psi^c(x):=\Psi(c)^{-1}\Psi(cx).
$$
A blow-down with varying basepoints is a cone limit of maps $\Psi^{c_j}$ at scales tending to infinity. Let $\mathscr G_{\mathrm{pt}}(\Psi)$ consist of the linear maps $L:V_N\to V_M$ for which the graded homomorphism $\widehat L$ itself occurs as a blow-down with varying basepoints.

\begin{mainthm}\label{thm:main-pointed}
The set $\mathscr G_{\mathrm{pt}}(\Psi)$ is compact, and
\begin{equation}\label{eq:intro-pointed-convexification}
 \mathcal A(\Psi)=\operatorname{conv}\mathscr G_{\mathrm{pt}}(\Psi),
 \qquad
 \operatorname{Ext}\mathcal A(\Psi)
 \subseteq\mathscr G_{\mathrm{pt}}(\Psi).
\end{equation}
Consequently every extreme point of $\mathcal A(\Psi)$ extends to a graded $\Lambda$-bi-Lipschitz isomorphism. Finally, $\mathcal A(\Psi)=\{L\}$ if and only if every blow-down of $\Psi$ with varying basepoints is the same graded map $\widehat L$.
\end{mainthm}

Remark \ref{rem:moving-basepoints-essential} illustrates the distinction between fixed and varying basepoints. The extreme points are also constrained by the full cohomology ring.

\begin{mainthm}\label{thm:main-extreme-cohomology}
For every $L\in\operatorname{Ext}\mathcal A(\Psi)$, there is an isomorphism of graded unital real algebras
$$
 \mathcal I_L:H^*(M;\RR)\longrightarrow H^*(N;\RR)
$$
whose restriction to $H^1(M;\RR)=V_M^*$ is $L^*:V_M^*\to V_N^*$.
\end{mainthm}

This prescribed compatibility yields the following quantitative obstruction.

\begin{mainthm}\label{thm:main-multiplicative-gap}
There exist a torsion-free nilpotent group $\Gamma$, two finite symmetric generating sets $S,T$, and $\varepsilon>0$ such that the asymptotic cones of $(\Gamma,d_S)$ and $(\Gamma,d_T)$ are isometric, but every $(\Lambda,A)$-quasi-isometry between them satisfies
\begin{equation}\label{eq:uniform-multiplicative-gap}
 \Lambda\ge1+\varepsilon.
\end{equation}
\end{mainthm}

Shalom proved the quasi-isometry invariance of the Betti numbers of finitely generated nilpotent groups \cite[Theorem 1.2]{Shalom2004}. Sauer proved the invariance of their real cohomology rings \cite{Sauer2006}, and Gotfredsen--Kyed extended this to connected simply connected nilpotent Lie groups \cite{GotfredsenKyed21}. Antonelli--Young obtain an isomorphism of cohomology algebras by averaging the pullback under an ergodic limit of translates of a quasi-isometry \cite[Theorem 1.1 and Proposition 1.2]{AntonelliYoung25}. Their construction uses the same normalized translation action on an orbit closure and selects an ergodic invariant measure by extremality \cite[Section 3 and Lemma 3.3]{AntonelliYoung25}. Here we choose an ergodic invariant measure in the fiber over each prescribed extreme point of $\mathcal A(\Psi)$ and apply Gotfredsen--Kyed's induction to obtain the required degree-one part. 

To construct $\Gamma$ as in \Cref{thm:main-multiplicative-gap}, we use the rational Lie algebra
$$
 (0,0,0,12,14+23,0,15-34+36).
$$
Its associated graded algebra admits a graded automorphism $P$, while $P^*$ cannot be the degree-one part of a cohomology-ring automorphism. Theorems \ref{thm:main-attainment} and \ref{thm:main-extreme-cohomology}, together with compactness, give the uniform gap. Breuillard--Le Donne previously obtained two word metrics on one nilpotent group with isometric cones but no $(1,A)$-quasi-isometry; in their example the infimum of the possible multiplicative constants is nevertheless one \cite{BreuillardLeDonne2013}. Here the infimum is strictly larger than one. Positive multiplicative gaps were previously known in negatively curved settings \cite{LafontSchmidtVanLimbeek17}.

\emph{Organization.} Sections \ref{sec:preliminaries} and \ref{sec:transport-body} fix notation and construct $\mathcal A(\Psi)$. Section \ref{sec:attainment} proves that first layer maps of Pansu differentials of every fixed-basepoint blowdown is realized by an invariant mean. Section \ref{sec:pointed-blowdowns} proves exact pointed convexification  and isolates bounded Abelian defect as a singleton criterion which is functorial. Sections \ref{sec:extreme-cohomology} and \ref{sec:metric-gap} prove the cohomological selection theorem and the multiplicative constant gap. Some auxiliary details are collected in Appendix \ref{app:cone-compactness}.

\section{Preliminaries}\label{sec:preliminaries}

Throughout, the groups $N$ and $M$ are finitely generated and non-trivial.

\subsection{Quasi-isometries} Let $X$ and $Y$ be metric spaces. By a $(\Lambda, A)$ quasi-isometry between $X$ and $Y$ we mean a map $\Psi:X\to Y$ such that for all $x,x'\in X$
$$
\Lambda^{-1}d(x,x')-A\le d(\Psi(x),\Psi(x'))\le\Lambda d(x,x')+A
$$
and is coarsely surjective, that is there exists $R>0$ such that for any $y\in Y$ there exists $x\in X$ with $d(\Psi(x),y)\le R$. For quasi-isometries between groups $N$ and $M$ we normalize $\Psi(e_N)=e_M$ via left-translating the map by $\Psi(e_N)^{-1}$. This preserves the above quasi-isometry constants.

\subsection{Asymptotic cones and their first layers}
Let $N$ be a finitely generated torsion-free nilpotent group with finite symmetric generating set $S_N$ and word metric $d_N$. Let $\widetilde N$ be the Malcev completion of $N$ with Lie algebra $\mathfrak n$ and exponential map $\exp_{\tilde N}$. Since $\mathfrak n$ is a nilpotent Lie algebra, there exists a smallest non-negative integer $s$ such that $\mathfrak n^{s+1}=\{0\}$, where we define $\mathfrak n^1:=\mathfrak n$ and $\mathfrak n^{i+1}:=[\mathfrak n, \mathfrak n^{i}]$ for $i\ge 1$. We can choose a vector-space decomposition
$$
\mathfrak n=W_1\oplus\cdots\oplus W_s
$$
so that $\mathfrak n^i=W_i\oplus\cdots\oplus W_s$. Put
$$
V_i^N:=\mathfrak n^i/\mathfrak n^{i+1},
\qquad
q_i:W_i\longrightarrow V_i^N
$$
for the quotient isomorphisms. Thus $\mathfrak n_\infty:=\bigoplus_{i=1}^sV_i^N$ is the associated graded Lie algebra, and its simply connected Lie group is the asymptotic cone $N_\infty$ of $N$. Its Carnot dilations are
$$
\delta_t\left(\exp_\infty(Z_1+\cdots+Z_s)\right)
:=\exp_\infty(tZ_1+t^2Z_2+\cdots+t^sZ_s),
\qquad Z_i\in V_i^N,\qquad t>0,
$$
where $\exp_\infty$ denotes the exponential map of $N_\infty$. We shall write $d_\infty$ when the group is clear from context. The first layer $V_1^N$ can be identified with
$$
V_N=N_{\mathrm{ab}}\otimes_{\mathbb Z}\mathbb R.
$$
Indeed, let $q:\mathfrak n\to V_1^N$ be the quotient map. The Baker--Campbell--Hausdorff formula gives
$$
q(\log_{\widetilde N}(xy))
=q(\log_{\widetilde N}x)+q(\log_{\widetilde N}y),
$$
since every remaining term is a commutator. Thus $n\mapsto q(\log_{\widetilde N}n)$ induces a linear map
$$
T:V_N\longrightarrow V_1^N,
\qquad
T\bigl((n[N,N])\otimes t\bigr)=t\,q(\log_{\widetilde N}n).
$$

To prove that $T$ is an isomorphism, let $\chi:N\to\mathbb R$ be any homomorphism. Malcev's extension theorem gives a unique Lie group homomorphism $\widetilde\chi:\widetilde N\to\mathbb R$ extending $\chi$; see \cite[Theorem 2.11]{Raghunathan72}. Its differential at the identity annihilates $[\mathfrak n,\mathfrak n]$, so it induces a linear functional $\lambda\in(V_1^N)^*$ satisfying
$$
\chi(n)=\widetilde\chi(n)
=\lambda\bigl(q(\log_{\widetilde N}n)\bigr).
$$
Conversely, every $\lambda\in(V_1^N)^*$ defines a Lie group homomorphism $\lambda\circ q\circ\log_{\widetilde N}$, and uniqueness in Malcev's theorem makes the functional extending a given $\chi$ unique. Homomorphisms $N\to\mathbb R$ are exactly the linear functionals on $V_N$, so the dual map $T^*:(V_1^N)^*\to V_N^*$ is bijective. Both spaces are finite dimensional, hence $T$ is an isomorphism.

We write $\pi_N:N\to V_N$ for the canonical projection through the abelianization. The identification $\iota_N:\widetilde N\to N_\infty$ is the diffeomorphism defined by
$$
\iota_N\left(
 \exp_{\widetilde N}(X_1+\cdots+X_s)
\right)
:=
\exp_\infty\left(q_1(X_1)+\cdots+q_s(X_s)\right),
\qquad X_i\in W_i.
$$
It is only a coordinate identification and need not be a group homomorphism. Define the rescaling map 
$$
\sigma_r^N:=\delta_{1/r}\circ\iota_N,
$$
and we use the customary shorthand $\delta_{1/r}(x):=\sigma_r^N(x)$ for $x\in N$.

Define $\pi_N^\infty:N_\infty\longrightarrow V_N$ by
$$
\pi_N^\infty\left(
 \exp_\infty(Z_1+\cdots+Z_s)
\right):=Z_1,
\qquad Z_i\in V_i^N.
$$
Since every element of the form $[X,Y]$ lies in $\bigoplus_{i\ge2}V_i^N$, it is easy to see that $\pi_N^\infty$ is a group homomorphism. Moreover,
$$
\pi_N^\infty\circ\delta_t=t\pi_N^\infty.
$$
It agrees with $\pi_N$ in the sense that
$$
\pi_N^\infty(\iota_N(n))=\pi_N(n),
\qquad
\pi_N^\infty(\sigma_r^N(n))=\frac1r\pi_N(n)
\qquad(n\in N).
$$
For a second finitely generated torsion-free nilpotent group $M$ with finite symmetric generating set $S_M$ and word metric $d_M$, put $V_M:=M_{\mathrm{ab}}\otimes_{\mathbb Z}\mathbb R$ and use the analogous notation $\pi_M,M_\infty,d_{M,\infty},\iota_M,\sigma_r^M$ and $\pi_M^\infty:M_\infty\to V_M$.

We can associate to the pair $(N,S_N)$ a norm $\ell_{S_N}$ on $V_1^N$ whose unit ball is the convex hull of $\pi_N(S_N)$ (called the Pansu limit norm for the pair $(N,S_N)$). The sub-Finsler metric $d_\infty$ is defined by the norm $\ell_{S_N}$ on $V_N$ in the following way: if $\gamma$ is an absolutely continuous horizontal curve and
$$
u_\gamma(t):=
\bigl(L_{\gamma(t)^{-1}}\bigr)_*\dot\gamma(t)\in V_N
$$
is its left-translated horizontal velocity, then
$$
\operatorname{Length}_{d_\infty}(\gamma)
=\int_0^1\ell_{S_N}(u_\gamma(t))\,dt.
$$
For $x,y\in N_\infty$, $d_{N,\infty}(x,y)$ is defined to be the infimum of the lengths of all absolutely continuous horizontal curves connecting $x$ to $y$. The metric $d_{N,\infty}$ is left-invariant and homogeneous:
$$
d_{N,\infty}(\delta_t(x),\delta_t(y))=t\,d_{N,\infty}(x,y) \quad (x,y\in N_\infty, t>0),
$$
and is a proper geodesic metric. It can be verified that $\pi_N^\infty:(N_\infty,d_{N,\infty})\to(V_N,\ell_{S_N})$ is $1$-Lipschitz, and $d_{N,\infty}(e,\exp_\infty v)=\ell_{S_N}(v)$ for $v\in V_N$. Analogous statements also hold for $\pi_M^\infty$ and $d_{M,\infty}$.

We record the coordinate form of Pansu convergence used below. For $G\in\{N,M\}$ and for every $C>0$,
\begin{equation}\label{eq:coordinate-pansu-convergence}
\sup_{\substack{a,b\in G\\ |a|_{S_G},\,|b|_{S_G}\le Cr}}
\left|
\frac{d_G(a,b)}r
-d_{G,\infty}\bigl(\sigma_r^G(a),\sigma_r^G(b)\bigr)
\right|\to0
\text{ as }r\to\infty.
\end{equation}
Moreover, for every $R>0$,
\begin{equation}\label{eq:cone-density}
\sup_{z\in\overline {B_{G,\infty}(e,R)}}
d_{G,\infty}\bigl(z,\sigma_r^G(G)\bigr)
\to0
\text{ as }r\to\infty.
\end{equation}
These are the coordinate forms of Pansu's theorem \cite{Pansu83} (see also \cite[Theorems 1 and 2]{BreuillardLeDonne2013}).

We shall also use the properness estimate: for every $R>0$ there is $C_R>0$ such that, for $r\ge1$,
\begin{equation}\label{eq:cone-bounded-lift}
\sigma_r^G(a)\in\overline {B_{G,\infty}(e,R)}
\quad\Longrightarrow\quad
|a|_{S_G}\le C_Rr.
\end{equation}
The proofs of \eqref{eq:cone-density} and \eqref{eq:cone-bounded-lift} are recalled in Subsection \ref{app:cone-density-details}.

Finally, write $B_R:=\{n\in N:|n|_{S_N}\le R\}$. Then \cite{Pansu83} and \cite[Corollary 9]{BreuillardLeDonne2013} give $|B_R|=c_{\mathrm{vol}}R^{Q_N}(1+o(1))$, where $Q_N$ is the homogeneous dimension of $N_\infty$. Hence, the balls $B_R$ are F{\o}lner with 
\begin{equation}\label{eq:balls-Folner}
\bigl(|B_{R+k}|-|B_{R-k}|\bigr)/|B_R|\to0   
\end{equation}
whenever $k=o(R)$.

\section{The invariant-mean linearization set of a quasi-isometry}\label{sec:transport-body}

Let $N,M$ be finitely generated torsion-free nilpotent groups with the word metrics and notation fixed in Section \ref{sec:preliminaries}. Fix a normalized $(\Lambda,A)$-quasi-isometry $\Psi:N\to M$.

Let $\omega(n,x):=\Psi(x)^{-1}\Psi(xn)$ be the displacement cocycle and set
$$
 \Phi_n(x):=\pi_M(\omega(n,x))
 =\pi_M(\Psi(xn))-\pi_M(\Psi(x)),
 \qquad n,x\in N.
$$
We give $V_M$ the norm $\ell_{S_M}$ and bounded $V_M$-valued functions the induced supremum norm. Then $\|\Phi_n\|_\infty\le\Lambda|n|_{S_N}+A$ and
\begin{equation}\label{eq:cocycle-identity-abelianized-displacement-function}
\Phi_{n_1n_2}(x)=\Phi_{n_1}(x)+\Phi_{n_2}(xn_1).
\end{equation}

\subsection{Invariant means and localization}

A \emph{right-invariant mean} on $N$ is a positive normalized functional $\mathsf A$ on $\ell^\infty(N)$ with $\mathsf A[f(\cdot g)]=\mathsf A[f]$ for all $g\in N$, $f\in \ell^\infty(N)$. Such an invariant mean exists since $N$ is amenable. For bounded $f:N\to V$ into a finite-dimensional real normed space, $\mathsf A[f]$ is defined componentwise, and
\begin{equation}\label{eq:mean-in-hull}
\mathsf A[f]\in\overline{\mathrm{conv}}f(N).
\end{equation}  
Indeed, if $\mathsf A[f]$ were outside $C:=\overline{\mathrm{conv}}f(N)$, Hahn-Banach separation theorem would give $\beta\in V^*$ with $\beta(\mathsf A[f])>\sup_{z\in C}\beta(z)$. But $\beta(\mathsf A[f])=\mathsf A[\beta\circ f]\le\sup_N(\beta\circ f)$ by positivity and normalization, a contradiction.

It is easy to see that if there exists $w_1,\dots,w_n\in V$ such that 
$$
f(x)=\sum_{i=1}^n c_i(x)w_i \quad(x\in N, c_1,\dots,c_n\in\ell^\infty(N)),
$$ 
then $\mathsf A[f]=\sum_{i=1}^n \mathsf A[c_i]w_i$.

Because $\Phi$ satisfies the cocycle identity \eqref{eq:cocycle-identity-abelianized-displacement-function}, $n\mapsto\mathsf A[\Phi_n]$ is a homomorphism $N\to V_M$ (using right-invariance of $\mathsf A$). It kills torsion and commutators, hence factors through $\pi_N$ and extends uniquely to a linear map
$$
L_{\mathsf A}:V_N\to V_M,\qquad L_{\mathsf A}\pi_N(n)=\mathsf A[\Phi_n] \qquad(n\in N),
$$
which is the invariant-mean linearization associated with $\mathsf A$.

\begin{definition}[Invariant-mean linearizations]\label{def:transport-body}
The \emph{invariant-mean linearization set} of $\Psi$ is
$$
\mathcal A(\Psi):=\{L_{\mathsf A}:\mathsf A \text{ a right-invariant mean on } N\}
\subseteq\Hom(V_N,V_M).
$$
\end{definition}
In what follows, we record some properties of the set $\mathcal A(\Psi)$.

\begin{proposition}[Bounded-distance invariance]
\label{prop:bounded-distance-invariance}
If $\Psi':N\to M$ is at bounded distance from $\Psi$, then
$$
 \mathcal A(\Psi')=\mathcal A(\Psi).
$$
\end{proposition}

\begin{proof}
The map $e(x):=\pi_M(\Psi'(x))-\pi_M(\Psi(x))$ is bounded. If $\Phi'_n$ is the displacement function of $\Psi'$, then for $n,x\in N$,
$$
 \Phi'_n(x)-\Phi_n(x)=e(xn)-e(x), \text{ i.e. } \Phi'_n-\Phi_n=n.e-e.
$$
Every right-invariant mean annihilates the term on the right, so it gives the same linearization for $\Psi'$ and $\Psi$.
\end{proof}

\begin{lemma}[Convexity and compactness]\label{lem:linearization-compact-convex}
$\mathcal A(\Psi)$ is a non-empty, compact and convex set.
\end{lemma}
\begin{proof}
The right-invariant means form a nonempty convex set, weak-* compact in $\ell^\infty(N)^\ast$, and $\mathsf A\mapsto L_{\mathsf A}$ is an affine and weak-* continuous map into $\Hom(V_N,V_M)$, so $\mathcal A(\Psi)$ is nonempty, convex, and compact.   
\end{proof}

\begin{lemma}[Scale invariance]\label{lem:scale-invariance}
For every right-invariant mean $\mathsf A$, every $n\in N$ and integer $k\ge1$,
$$
L_{\mathsf A}\pi_N(n)=\mathsf A[Q_{k,n}],\qquad Q_{k,n}:=\frac1k\Phi_{n^k}.
$$
\end{lemma}

\begin{proof}
$n\mapsto\mathsf A[\Phi_n]$ is a homomorphism into $(V_M,+)$, so $\mathsf A[\Phi_{n^k}]=k\mathsf A[\Phi_n]=kL_{\mathsf A}\pi_N(n)$.
\end{proof}

\begin{proposition}[The universal norm upper bound]\label{prop:localization}
For every right-invariant mean $\mathsf A$ and $n\in N$,
\begin{equation}\label{eq:localization}
L_{\mathsf A}\pi_N(n)\in\bigcap_{k\in \NN}\overline{\mathrm{conv}}\{Q_{k,n}(x):x\in N\}.
\end{equation}
Consequently every $L\in\mathcal A(\Psi)$ satisfies $\ell_{S_M}(Lv)\le\Lambda\ell_{S_N}(v)$ for all $v\in V_N$.
\end{proposition}

\begin{proof}
The first part follows from Lemma \ref{lem:scale-invariance} with \eqref{eq:mean-in-hull} applied to $f=Q_{k,n}$ for each $k$. For the bound, since $\ell_{S_M}(\pi_M(w))\le|w|_{S_M}$, by left-invariance of $d_M$,
\begin{align*}
\ell_{S_M}(Q_{k,n}(x))=\frac1k\ell_{S_M}(\pi_M(\omega(n^k,x)))
&\le\frac1k|\Psi(x)^{-1}\Psi(xn^k)|_{S_M}
=\frac1k d_M(\Psi(x),\Psi(xn^k))\\
&\le\frac{\Lambda|n^k|_{S_N}+A}{k}.
\end{align*}
We claim that
\begin{equation}\label{eq:power-stable-length}
\lim_{k\to\infty}\frac{|n^k|_{S_N}}{k}=\ell_{S_N}(\pi_N(n)).
\end{equation}
Write $\log_{\tilde N} n=Y_1+\cdots+Y_s$ $(Y_i\in W_i)$ as in Section \ref{sec:preliminaries}, and put $X_i=q_i(Y_i)$. Then
$$
 \sigma_k^N(n^k)
 =\exp_\infty(X_1+k^{-1}X_2+\cdots+k^{1-s}X_s)
 \to\exp_\infty(X_1) \text{ as }k\to\infty,
$$
where $X_1=\pi_N(n)$. Metric convergence and the identity $d_\infty(e,\exp_\infty v)=\ell_{S_N}(v)$ prove \eqref{eq:power-stable-length}. Hence, for a fixed $\varepsilon>0$, the convex hulls in \eqref{eq:localization} are contained in the $\ell_{S_M}$-ball of radius $\Lambda\ell_{S_N}(\pi_N(n))+\varepsilon$ for large enough $k$. This gives the bound on rational multiples of $\pi_N(n)$, and therefore on all of $V_N$ by density of rationals and continuity.
\end{proof}

\section{Realization of Pansu first-layer differentials}
\label{sec:attainment}

Throughout this section, $N,M,\Psi,\Lambda,A,\ell_{S_N}, \ell_{S_M}$ are as in Section \ref{sec:transport-body}.

We record the following standard compactness statement; compare \cite[Lemma 10.83]{DrutuKapovich2018}. Its proof is given in Subsection \ref{app:cone-map-proof}.

\begin{lemma}[Cone maps]\label{lem:cone-map}
For every $r_j\to\infty$ there is a subsequence, not relabeled, and a $\Lambda$-bi-Lipschitz homeomorphism $\Psi_\infty:N_\infty\to M_\infty$ with $\Psi_\infty(e)=e$ such that
\begin{equation}\label{eq:uniform-conv}
\gamma_j(R):=\sup_{x\in D_j^N\cap B_{N,\infty}(e,R)}d_{M,\infty}\bigl(\Psi_j(x),\Psi_\infty(x)\bigr)\longrightarrow0 \qquad(R\ge1 \text{ fixed}),
\end{equation}
where
$$
D_j^N:=\sigma_{r_j}^N(N),\qquad D_j^M:=\sigma_{r_j}^M(M),
$$
and
$$
\Psi_j:=\sigma_{r_j}^M\circ\Psi\circ(\sigma_{r_j}^N)^{-1}:D_j^N\longrightarrow D_j^M.
$$
\end{lemma}

\begin{definition}[Pansu differentiability point]\label{def:pansu-differentiability-point}
Let $F:N_\infty\to M_\infty$ be Lipschitz. A point $p\in N_\infty$ is a \emph{Pansu differentiability point} of $F$ if there is a graded homomorphism $D_pF:N_\infty\to M_\infty$ such that
$$
\delta_{1/t}\left(F(p)^{-1}F(p\delta_t(g))\right)
\longrightarrow D_pF(g)
$$
uniformly for $g$ on compact subsets of $N_\infty$ as $t\to0^+$.
\end{definition}

\begin{lemma}[First layers of differentials]\label{lem:first-layer-biLip}
The map $\Psi_\infty$ is Pansu differentiable at almost every $p\in N_\infty$. At every such point, $D_p\Psi_\infty$ is a graded homomorphism. If $D^1_p:V_N\to V_M$ denotes its first-layer map, then
\begin{equation}\label{eq:first-layer-norm-bound}
\Lambda^{-1}\ell_{S_N}(v)\le\ell_{S_M}(D^1_pv)\le\Lambda\ell_{S_N}(v)\qquad(v\in V_N).
\end{equation}
Moreover $D_p\Psi_\infty$ is a graded isomorphism. In particular $D^1_p:V_N\to V_M$ is a linear isomorphism.
\end{lemma}

\begin{proof}
Differentiability and the graded-homomorphism property are Pansu's Rademacher theorem \cite[Theorem 2]{Pansu89}. We already know that $\Psi_\infty$ is globally $\Lambda$-bi-Lipschitz. Fix $g\in N_\infty$. By left-invariance and homogeneity of $d_\infty$, $d_\infty(p\delta_t(g),p)=t d_\infty(e,g)$, so
$$
\Lambda^{-1}td_\infty(e,g)\le d_\infty(\Psi_\infty(p\delta_t(g)),\Psi_\infty(p))
\le\Lambda td_\infty(e,g).
$$
Dividing by $t$ and letting $t\to0$, the middle term tends to $d_\infty(e,D_p\Psi_\infty(g))$, giving
\begin{equation}\label{eq:full-differential-biLip}
\Lambda^{-1}d_\infty(e,g)\le d_\infty(e,D_p\Psi_\infty(g))\le\Lambda d_\infty(e,g) \qquad(g\in N_\infty).
\end{equation}
Thus $D_p\Psi_\infty$ is an injective graded homomorphism. Specializing \eqref{eq:full-differential-biLip} to $g=\exp v$ with $v\in V_N$ and using $D_p\Psi_\infty(\exp v)=\exp(D^1_pv)$ together with $d_\infty(e,\exp w)=\ell_{S_M}(w)$ for $w\in V_M$ gives \eqref{eq:first-layer-norm-bound}.

For surjectivity, since $\Psi_\infty$ is a bi-Lipschitz homeomorphism (Lemma \ref{lem:cone-map}), $N_\infty$ and $M_\infty$ have the same topological, hence Lie dimension. The injective graded homomorphism $D_p\Psi_\infty$ restricts injectively on every $V^N_i$, so $\sum_{i=1}^s \dim V^N_i=\sum_{i=1}^s\dim V^M_i$ forces $\dim V^N_i=\dim V^M_i$ for each $i$. Hence $D_p\Psi_\infty$ is a graded isomorphism, and in particular $D^1_p:V_N\to V_M$ is a linear isomorphism.
\end{proof}

\begin{lemma}[An approximation lemma]\label{lem:mesoscopic-approximation}
Let $\Psi_\infty$ be a cone map of $\Psi$ as in Lemma \ref{lem:cone-map} and $p$ a Pansu differentiability point of $\Psi_\infty$. Put $D^1:=D_p^1\Psi_\infty$. Suppose that $\Psi_\infty$ is obtained along $r_j\to\infty$. For every pair of finite nonempty subsets $E,F\subseteq N$, every $\varepsilon>0$, and every $R_0\ge1$, there is $J_0\ge1$ such that for every $j\ge J_0$ there are $c_j\in N$ and integers $\rho_j\ge R_0$ so that, with $\Omega_j=c_jB_{\rho_j}$,
\begin{align}
\max_{n\in E}\left\|\frac1{|\Omega_j|}\sum_{x\in\Omega_j}\Phi_n(x)-D^1\pi_N(n)\right\|&<\varepsilon, \label{eq:mesoscopic-cocycle}\\
\max_{g\in F}\frac{|\Omega_j g\mathbin\triangle\Omega_j|}{|\Omega_j|}
&<\varepsilon, \label{eq:mesoscopic-folner}
\end{align}
where $\|\cdot\|$ denotes the norm $\ell_{S_M}$ on $V_M$. Moreover,
\begin{equation}\label{eq:mesoscopic-position}
 \frac{\rho_j}{r_j}<\varepsilon, \qquad d_{N,\infty}(\sigma_{r_j}^N(c_j),p)<\varepsilon.
\end{equation}
\end{lemma}

\begin{proof}
Let $r_j\to\infty$ produce $\Psi_\infty$, write $D:=D_p\Psi_\infty$, $D^1:=D^1_p\Psi_\infty$, and for $t\in(0,1]$, $R\ge0$ put
\begin{equation}\label{eq:Pansu-differentiability-uniform}
\eta(t,R):=\sup_{\substack{g\in N_{\infty}\\d_\infty(e,g)\le R}} d_\infty\left(\delta_{1/t}(\Psi_\infty(p)^{-1}\Psi_\infty(p\delta_t(g))),D(g)\right),
\end{equation}
so differentiability at $p$ yields $\eta(t,R)\to0$ as $t\to0$ for each $R$. Choose once and for all $x_j\in N$ with $d_\infty(\sigma_{r_j}^N(x_j),p)\to0$; this choice is independent of the increment $n$ considered below.

For $h\in N_\infty$ with $d_\infty(e,h)\le R$ and $v\in V_N$ with $\ell_{S_N}(v)\le R$, and $q:=p\delta_t(h)$,
\begin{equation}\label{eq:cone-increment}
\bigl\|\pi_M^\infty(\Psi_\infty(q\exp_\infty(tv)))-\pi_M^\infty(\Psi_\infty(q))-tD^1v\bigr\|
\le2t\eta(t,2R).
\end{equation}
Indeed, by the definition of $\eta$, the homogeneity of $\pi_M^\infty$, and its $1$-Lipschitz property, whenever $g\in N_\infty$ with $d_\infty(e,g)\le2R$,
$$
 \left\|
 \pi_M^\infty\left(
  \Psi_\infty(p)^{-1}\Psi_\infty(p\delta_t(g))
 \right)-t\pi_M^\infty(D(g))
 \right\|
 \le t\eta(t,2R).
$$
Apply this with $g=h\exp_\infty v$ and with $g=h$. Here $q\exp_\infty(tv)=p\delta_t(h\exp_\infty v)$ because $v\in V^N_1$. Since $d_\infty(e,h\exp_\infty v)\le2R$ and both $D$ and $\pi_M^\infty$ are homomorphisms, we have
\begin{align*}
&\bigl\|\pi_M^\infty(\Psi_\infty(q\exp_\infty(tv)))-\pi_M^\infty(\Psi_\infty(q))-tD^1v\bigr\| \\
 &\le \left\| \pi_M^\infty(\Psi_\infty(p)^{-1}\Psi_\infty(p\delta_t(g)))-t\pi_M^\infty(D(g))\right\|+\left\|\pi_M^\infty(
  \Psi_\infty(p)^{-1}\Psi_\infty(p\delta_t(h)))-t\pi_M^\infty(D(h))\right\|\\
 &\le 2t\eta(t,2R).
\end{align*}

Fix $n\in N$, write $|n|:=|n|_{S_N}$, and put $v:=\pi_N(n)$. Choose $R\ge\max\{1,2|n|\}$ and $t\in(0,1]$, then set $k_j:=\lceil tr_j\rceil$ and $\rho_j:=\lceil Rtr_j\rceil$. Put $F_j:=x_jB_{\rho_j}$ and $E_j(n):=|F_j|^{-1}\sum_{x\in F_j}\Phi_n(x)$.

Observe that $\Phi_{n^k}(x)=\sum_{i=0}^{k-1}\Phi_n(xn^i)$, so
$$
\left\|E_j(n)-\frac{1}{|F_j|}\sum_{x\in F_j}\frac{1}{k_j}\Phi_{n^{k_j}}(x)\right\|
\le\|\Phi_n\|_\infty\frac{|B_{\rho_j+k_j|n|}|-|B_{\rho_j-k_j|n|}|}{|B_{\rho_j}|},
$$
since $F_jn^i\triangle F_j\subseteq x_j(B_{\rho_j+k_j|n|}\setminus B_{\rho_j-k_j|n|})$. By \eqref{eq:balls-Folner}, the ratio tends to $(1+|n|/R)^{Q_N}-(1-|n|/R)^{Q_N}\le C_{Q_N}|n|/R$ for a constant $C_{Q_N}$ depending only on $Q_N$, which implies
\begin{equation}\label{eq:telescoping-error}
\limsup_j\left\|E_j(n)-\frac{1}{|F_j|}\sum_{x\in F_j}\frac{1}{k_j}\Phi_{n^{k_j}}(x)\right\|
\le\frac{C_{Q_N}|n|\|\Phi_n\|_\infty}{R}.
\end{equation}
Write $x=x_jy\in F_j$, where $y\in B_{\rho_j}$, and put
$$
\sigma_j:=\sigma_{r_j}^N,
\qquad
q_x:=\sigma_j(x).
$$
Since $\sigma_{r_j}(x_j)\to p$, \eqref{eq:cone-bounded-lift} gives $C_p>0$ such that $|x_j|_{S_N}\le C_pr_j$ for all $j$. For large enough $j$, we have $\rho_j=\lceil Rtr_j\rceil\le(R+1)r_j$, and hence
$$
|x_jy|_{S_N}\le |x_j|_{S_N}+|y|_{S_N}
\le (C_p+R+1)r_j
\qquad(y\in B_{\rho_j}).
$$
Thus, for large enough $j$, $x_jB_{\rho_j}\subseteq B_{Cr_j}$, where $C:=C_p+R+1$ is independent of $j$. Applying \eqref{eq:coordinate-pansu-convergence} with $a=x_jy$ and $b=x_j$ shows that the supremum in
$$
\varepsilon_j:=
\sup_{y\in B_{\rho_j}}
\left|
d_\infty\bigl(\sigma_j(x_jy),\sigma_j(x_j)\bigr)
-\frac{|y|_{S_N}}{r_j}
\right|
+d_\infty(\sigma_j(x_j),p)+\frac1{r_j}
$$
tends to zero. The remaining two terms also tend to zero, so $\varepsilon_j\to0$.

Consequently, uniformly over $x\in F_j$, we have
\begin{align*}
d_\infty(q_x,p)
&\le
d_\infty\bigl(\sigma_j(x_jy),\sigma_j(x_j)\bigr)
+d_\infty(\sigma_j(x_j),p)\\
&\le
\frac{|y|_{S_N}}{r_j}+\varepsilon_j-\frac{1}{r_j}
\le Rt+\varepsilon_j.
\end{align*}
Define $h_x:=\delta_{1/t}(p^{-1}q_x).$ Then $q_x=p\delta_t(h_x)$ and, by homogeneity and left invariance,
$$
d_\infty(e,h_x)
=\frac1td_\infty(p,q_x)
\le R+\varepsilon_j/t.
$$
Here $t$ is fixed while $j\to\infty$, so $d_\infty(e,h_x)\le2R$ for all sufficiently large $j$, uniformly over $x\in F_j$. Write
$$
\log_{\widetilde N}n=Y_1+\cdots+Y_s,
\qquad Y_i\in W_i,
$$
and put $X_i:=q_i(Y_i)\in V_i^N$. Under the canonical identification of $V^N_1$ with $V_N$ one has
$$
X_1=\pi_N(n)=v.
$$
Since $n^{k_j}=\exp_{\widetilde N}(k_j\log_{\widetilde N}n)$,
$$
z_j:=\sigma_j(n^{k_j})
=
\exp_\infty\left(
 \frac{k_j}{r_j}X_1+
 \frac{k_j}{r_j^2}X_2+\cdots+
 \frac{k_j}{r_j^s}X_s
\right).
$$
Because $k_j/r_j\to t$ and $k_j/r_j^i\to0$ for every $i\ge2$,
$$
z_j\longrightarrow\exp_\infty(tv).
$$
Transport the multiplication of $\widetilde N$ to $N_\infty$ by defining
$$
u*_rv:=u*_r^Nv= 
\sigma_r^N\left(
 (\sigma_r^N)^{-1}(u)(\sigma_r^N)^{-1}(v)
\right).
$$
As in the proof of \eqref{eq:cone-density} in Subsection \ref{app:cone-density-details}, we have the locally uniform convergence
$$
*_r\to\text{the multiplication of }N_\infty.
$$
Now
$$
\sigma_j(xn^{k_j})=q_x*_{r_j}z_j.
$$
The points $q_x$, for $x\in F_j$, remain in a fixed compact subset of $N_\infty$, and the sequence $(z_j)$ also remains in a fixed compact subset. Choose a compact set $K$ containing all these points and $\exp_\infty(tv)$, and set
$$
\lambda_j(K):=
\sup_{u,w\in K}d_\infty(u*_{r_j}w,uw).
$$
Local uniform convergence of the products gives $\lambda_j(K)\to0$ as $j\to\infty$. Consequently, by left invariance,
$$
\begin{aligned}
d_\infty\left(
 \sigma_j(xn^{k_j}),q_x\exp_\infty(tv)
\right)
&=
d_\infty\left(
 q_x*_{r_j}z_j,q_x\exp_\infty(tv)
\right)\\
&\le
\lambda_j(K)
+d_\infty\left(z_j,\exp_\infty(tv)\right)
=:\beta_j.
\end{aligned}
$$
Since $\beta_j\to0$ and is independent of $x$, this convergence is uniform over $x\in F_j$.

All $z\in F_j\cup F_jn^{k_j}$ satisfy $|z|_{S_N}\le Cr_j$ for some $C\ge1$ independent of $j$. By \eqref{eq:coordinate-pansu-convergence}, after enlarging $C$ we also have $\sigma_j(z)\in B_{N,\infty}(e,C)$ for all sufficiently large $j$, so \Cref{lem:cone-map} applies. Put
$$
\Delta_j(x):=
\pi_M^\infty(\Psi_\infty(q_x\exp(tv)))
-\pi_M^\infty(\Psi_\infty(q_x)).
$$
Using $\pi_M(w)=r_j\pi_M^\infty(\delta_{1/r_j}w)$, the $1$-Lipschitz property of $\pi_M^\infty$, the $\Lambda$-Lipschitz property of $\Psi_\infty$, and the uniform bounds above, the triangle inequality gives
\begin{equation}\label{eq:discrete-cone-increment}
\bigl\|\Phi_{n^{k_j}}(x)-r_j\Delta_j(x)\bigr\|
\le r_j\bigl(2\gamma_j(C)+\Lambda\beta_j\bigr)
\qquad(x\in F_j).
\end{equation}
By \eqref{eq:cone-increment}, applied with parameter $2R$, and $tr_j/k_j\to1$, $r_j/k_j\to1/t$,
\begin{align*}
\Bigl\|\frac{1}{k_j}\Phi_{n^{k_j}}(x)-D^1v\Bigr\| 
&\le \frac1{k_j}
 \bigl\|\Phi_{n^{k_j}}(x)-r_j\Delta_j(x)\bigr\|\\
&\quad+\frac{r_j}{k_j}\|\Delta_j(x)-tD^1v\|
 +\left|\frac{tr_j}{k_j}-1\right|\|D^1v\|\\
&\le \frac{r_j}{k_j}(2\gamma_j(C)+\Lambda\beta_j)
 +\frac{2tr_j}{k_j}\eta(t,4R)\\
&\quad+\left|\frac{tr_j}{k_j}-1\right|\|D^1v\|,
\end{align*}
uniformly over $x\in F_j$. Combining with \eqref{eq:telescoping-error}, we get
\begin{equation}\label{eq:main-window-estimate}
\limsup_j\|E_j(n)-D^1v\|\le\frac{C_{Q_N}|n|\|\Phi_n\|_\infty}{R}+2\eta(t,4R).
\end{equation}
Now fix finite subsets $E,F\subseteq N$ and $\varepsilon>0$. Choose $R\ge\max\{1,2\max_{n\in E}|n|_{S_N}\}$ so large that
$$
\frac{C_{Q_N}|n|\|\Phi_n\|_\infty}{R}<\frac\varepsilon4
\qquad(n\in E).
$$
Then choose $t\in(0,1]$ so small that
$$
 2\eta(t,4R)<\varepsilon/4
 \qquad\text{and}\qquad Rt<\varepsilon/2.
$$
The estimate \eqref{eq:main-window-estimate}, applied to the finitely many elements of $E$, shows that \eqref{eq:mesoscopic-cocycle} holds with $c_j=x_j, \Omega_j=F_j$ for all sufficiently large $j$. Since $\rho_j=\lceil Rtr_j\rceil\to\infty$, \eqref{eq:balls-Folner} gives, uniformly for $g\in F$,
$$
\frac{|F_jg\mathbin\triangle F_j|}{|F_j|}
\le
\frac{|B_{\rho_j+|g|_{S_N}}|-|B_{\rho_j-|g|_{S_N}}|}{|B_{\rho_j}|}
\longrightarrow0.
$$
Thus there is $J_0$ such that, for every $j\ge J_0$,
$$
 \rho_j\ge R_0,\qquad
 \frac{\rho_j}{r_j}<\varepsilon,\qquad
 d_{N,\infty}(\sigma_{r_j}^N(x_j),p)<\varepsilon,
$$
which proves \eqref{eq:mesoscopic-folner} and \eqref{eq:mesoscopic-position}.
\end{proof}

\begin{theorem}[Averages on translated balls]
\label{thm:translated-ball-realization}
Let $\Psi_\infty$ be a cone map obtained along $r_j\to\infty$, and let $p$ be a Pansu differentiability point of $\Psi_\infty$. There are centers $c_j\in N$ and integers $\rho_j\to\infty$ such that
\begin{equation}\label{eq:geometric-limits}
 \frac{\rho_j}{r_j}\longrightarrow0,
 \qquad
 \sigma_{r_j}^N(c_j)\longrightarrow p,
\end{equation}
and, for every $n,g\in N$,
\begin{align}
 \frac1{|B_{\rho_j}|}\sum_{x\in c_jB_{\rho_j}}\Phi_n(x)
 &\longrightarrow D_p^1\Psi_\infty\,\pi_N(n),
 \label{eq:concrete-realization}\\
 \frac{|c_jB_{\rho_j}g\mathbin\triangle c_jB_{\rho_j}|}
 {|B_{\rho_j}|}&\longrightarrow0.
 \label{eq:concrete-folner}
\end{align}
\end{theorem}

\begin{proof}
Enumerate $N=\{n_1,n_2,\ldots\}$. For each $k\ge1$, take $E=F=\{n_1,\ldots,n_k\}$, $\varepsilon=k^{-1}$ and $R_0=k$ in Lemma \ref{lem:mesoscopic-approximation}. Then, there exists $J_k\ge 1$ such that for all $j\ge J_k$ there exists $c_{k,j}\in N$ and $\rho_{k,j}\ge k$ such that 
\begin{align*}
\max_{n\in E}
\left\|\frac1{|\Omega_{k,j}|}\sum_{x\in\Omega_{k,j}}\Phi_n(x)
-D^1\pi_N(n)\right\|&<\varepsilon,
\\
\max_{g\in F}\frac{|\Omega_{k,j} g\mathbin\triangle\Omega_{k,j}|}{|\Omega_{k,j}|}
&<\varepsilon,\\
 \frac{\rho_{k,j}}{r_j}<\varepsilon,
 \qquad
 d_{N,\infty}(\sigma_{r_j}^N(c_{k,j}),p)&<\varepsilon.
\end{align*}
where $\Omega_{k,j}=c_{k,j}B_{\rho_{k,j}}$. Without loss of generality, we can choose $(J_k)$ to be strictly increasing. For $j\ge J_1$, put $m(j):=\max\{k:J_k\le j\}$ and define $c_j:=c_{m(j),j},\rho_j:=\rho_{m(j),j}$ while choosing the finitely many earlier terms arbitrarily. Since $m(j)\to\infty$ as $j\to\infty$, \eqref{eq:mesoscopic-position} gives \eqref{eq:geometric-limits} and $\rho_j\ge m(j)\to\infty$. Every fixed $n,g\in N$ lies among $n_1,\ldots,n_{m(j)}$ for all sufficiently large $j$, so \eqref{eq:mesoscopic-cocycle} and \eqref{eq:mesoscopic-folner} give the remaining limits.
\end{proof}

Let $\mathscr D_1(\Psi)$ denote the set of $D_p^1\Psi_\infty$ over all fixed-basepoint blow-downs $\Psi_\infty$ of $\Psi$ and all their Pansu differentiability points $p$.

\begin{corollary}[$\mathscr D_1(\Psi)$ are invariant-mean linearizations]\label{cor:mean-realization}
Every element of $\mathscr D_1(\Psi)$ is an invariant-mean linearization. Consequently
$$
 \overline{\operatorname{conv}}\mathscr D_1(\Psi)
 \subseteq\mathcal A(\Psi)
 \subseteq
 \bigl\{L:\ell_{S_M}(Lv)\le\Lambda\ell_{S_N}(v)
                  \text{ for all }v\in V_N\bigr\}.
$$
In particular, $\mathcal A(\Psi)$ contains a linear isomorphism $L$ with
$$
 \Lambda^{-1}\ell_{S_N}(v)
 \le\ell_{S_M}(Lv)
 \le\Lambda\ell_{S_N}(v)
 \qquad(v\in V_N).
$$
\end{corollary}

\begin{proof}
Let $c_j,\rho_j$ be as in Theorem \ref{thm:translated-ball-realization} and define a sequence of positive linear functionals $(A_j)$ on $\ell^\infty(N)$ as follows: 
$$
A_j(f)=\frac{1}{|c_jB_{\rho_j}|}\sum_{x\in c_jB_{\rho_j}}f(x) \quad (f\in\ell^\infty(N)).
$$
Then, $(A_j)$ lies in the closed unit ball of $\ell^\infty(N)^\ast$ which is weak-* compact by Banach-Alaoglu theorem. Hence, there exists a weak-* convergent subnet of $(A_j)$ with limit $\mathsf A$. $\mathsf A$ is right invariant due to \eqref{eq:concrete-folner}, while \eqref{eq:concrete-realization} gives $L_{\mathsf A}=D_p^1\Psi_\infty$. Since $\mathcal A(\Psi)$ is compact and convex, it contains the closed convex hull of $\mathscr D_1(\Psi)$ and the other inclusion is Proposition \ref{prop:localization}. Finally, Lemma \ref{lem:cone-map} supplies a blow-down, Pansu's theorem supplies a differentiability point, and Lemma \ref{lem:first-layer-biLip} gives the two-sided bound and invertibility.
\end{proof}

The first layer map of a graded homomorphism between Carnot groups determines the homomorphism, since the first layer generates the Lie algebra. We use the following corresponding rigidity fact for Lipschitz maps.

\begin{lemma}[Constant Pansu differential]
\label{lem:constant-pansu-differential}
Let $G,H$ be Carnot groups and let $F:G\to H$ be Lipschitz. Suppose that there is a linear map $L:V_1^G\to V_1^H$ such that
$$
 D_p^1F=L
$$
for almost every $p\in G$. Then $L$ extends to a graded homomorphism $\widehat L:G\to H$, and
$$
 F(x)=F(e)\widehat L(x)
 \qquad(x\in G).
$$
\end{lemma}

\begin{proof}
At almost every $p$, $D_pF$ is a graded homomorphism and is uniquely determined by its restriction to the first layer. Thus the differentials in the hypothesis equal one graded homomorphism $\widehat L$.

Fix $X\in V_1^G$. Right translations preserve Haar measure, so Fubini's theorem on bounded $t$-intervals shows that, for Haar-almost every $q$, the point $q\exp(tX)$ is a differentiability point with differential $\widehat L$ for almost every $t$. The curve $\gamma_q:t\mapsto F(q\exp(tX))$ is Lipschitz and hence absolutely continuous in Lie group coordinates on bounded intervals. For every such $q$, Pansu differentiability gives, for almost every $t$,
$$
\left(L_{\gamma_q(t)^{-1}}\right)_\ast\dot\gamma_q(t)=L(X).
$$
Thus $\dot\gamma_q(t)=(L_{\gamma_q(t)})_\ast L(X)$ almost everywhere, with $\gamma_q(0)=F(q)$. Uniqueness for absolutely continuous solutions of this smooth left-invariant differential equation gives $\gamma_q(t)=F(q)\exp(tL(X))$ for every $t$. The set of such $q$ is dense; continuity of $F$ and of multiplication extends this identity to every $q$. Consequently, we have
\begin{equation}\label{eq:uniqueness_IVP}
 F(q\exp(tX))=F(q)\exp(tL(X))
\end{equation}
for every $q$ and $t$.

Let $x\in G$. By the Chow--Rashevsky theorem, write $x=\prod_{i=1}^k\exp(X_i)$ with $X_i\in V_1^G$. Repeated application of \eqref{eq:uniqueness_IVP} gives
$$
F(x)=F(e)\prod_{i=1}^k\exp(L(X_i))=F(e)\widehat L(x).
$$

\end{proof}

\begin{theorem}[Fixed-basepoint singleton rigidity]
\label{thm:singleton-rigidity}
If $\mathcal A(\Psi)=\{L\}$, then $L$ extends to a graded $\Lambda$-bi-Lipschitz isomorphism $\widehat L:N_\infty\to M_\infty$, and every fixed-basepoint blow-down of $\Psi$ equals $\widehat L$.
\end{theorem}

\begin{proof}
Let $\Psi_\infty$ be any fixed-basepoint blow-down. By Corollary \ref{cor:mean-realization}, every first-layer Pansu differential of $\Psi_\infty$ belongs to $\mathcal A(\Psi)$ and hence equals $L$. Pansu differentiability holds almost everywhere, so Lemma \ref{lem:constant-pansu-differential} and $\Psi_\infty(e)=e$ give $\Psi_\infty=\widehat L$. Since each fixed-basepoint blow-down is $\Lambda$-bi-Lipschitz, so is $\widehat L$.
\end{proof}

\begin{proof}[Proof of \Cref{thm:main-attainment}]
The first part is due to Theorem \ref{thm:translated-ball-realization}. Compactness and convexity of $\mathcal A(\Psi)$ is by Lemma \ref{lem:linearization-compact-convex}. Proposition \ref{prop:bounded-distance-invariance} shows that $\mathcal A(\Psi)$ is unchanged when $\Psi$ is replaced by a map at bounded distance from it, and \eqref{eq:intro-sandwich} is by Corollary \ref{cor:mean-realization}. The final assertion is Theorem \ref{thm:singleton-rigidity}.
\end{proof}

\section{Pointed blow-downs and exact convexification}
\label{sec:pointed-blowdowns}

The blow-downs considered so far keep the identity as basepoint. We now allow the basepoint to move. This records the large-scale behavior of $\Psi$ uniformly over the whole group. With this enlarged family, the one-sided fixed-basepoint statement in \eqref{eq:intro-sandwich} becomes an exact convex-hull formula.

For $c\in N$, define the normalized translate
\begin{equation}\label{eq:normalized-translate}
 \Psi^c(x):=\Psi(c)^{-1}\Psi(cx),\qquad x\in N.
\end{equation}
Every $\Psi^c$ fixes the identity and is a $(\Lambda,A)$-quasi-isometry with the same coarse-surjectivity constant as $\Psi$. Its displacement functions are
\begin{equation}\label{eq:translated-displacement}
 \Phi_n^c(x)
 :=\pi_M\bigl(\Psi^c(x)^{-1}\Psi^c(xn)\bigr)
 =\Phi_n(cx) \quad(n\in N).
\end{equation}

For $c\in N$ and $r>0$, put
$$
 D_r^N:=\sigma_r^N(N),\qquad D_r^M:=\sigma_r^M(M),
$$
and define 
\begin{equation}\label{eq:rescaled-normalized-translate}
 \Psi_{c,r}:=
 \sigma_r^M\circ\Psi^c\circ(\sigma_r^N)^{-1}
 :D_r^N\longrightarrow D_r^M.
\end{equation}
Thus, for $n\in N$,
\begin{equation}\label{eq:rescaled-normalized-translate-explicit}
 \Psi_{c,r}\bigl(\sigma_r^N(n)\bigr)
 =\sigma_r^M\bigl(\Psi(c)^{-1}\Psi(cn)\bigr).
\end{equation}

\begin{definition}[Pointed blow-down]\label{def:pointed-blowdown}
A continuous map $F:N_\infty\to M_\infty$ is a \emph{pointed blow-down} of $\Psi$ if there are sequences $c_j\in N$ and $r_j\to\infty$ such that, for every fixed $R\ge1$,
\begin{equation}\label{eq:pointed-uniform-conv}
 \gamma_{j,F}(R):=
 \sup_{z\in D_{r_j}^N\cap\overline B_{N,\infty}(e,R)}
 d_{M,\infty}\bigl(\Psi_{c_j,r_j}(z),F(z)\bigr)
 \longrightarrow0.
\end{equation}
The set of all such maps is denoted by $\mathscr B_{\mathrm{pt}}(\Psi)$. In particular, $F(e)=e$, and taking $c_j=e$ recovers the fixed-basepoint blow-downs.

The associated set of first-layer differentials is
$$
 \mathscr D^1_{\mathrm{pt}}(\Psi)
 :=\bigl\{D_p^1F:
 F\in\mathscr B_{\mathrm{pt}}(\Psi),\quad
 p\text{ is a Pansu differentiability point of }F\bigr\}.
$$
Finally, $\mathscr G_{\mathrm{pt}}(\Psi)$ is the set of maps $L\in\Hom(V_N,V_M)$ for which there is a graded homomorphism $\widehat L:N_\infty\to M_\infty$ whose first-layer map is $L$ and such that $\widehat L\in\mathscr B_{\mathrm{pt}}(\Psi)$.
\end{definition}

For any prescribed sequences $c_j\in N$ and $r_j\to\infty$, the proof of Lemma \ref{lem:cone-map} applies to the maps $\Psi^{c_j}$, which fix the identity and have common quasi-isometry and coarse-surjectivity constants. It therefore gives a subsequence satisfying \eqref{eq:pointed-uniform-conv}. In particular, pointed blow-downs exist, and every pointed blow-down is a $\Lambda$-bi-Lipschitz homeomorphism.

\begin{lemma}[Differentials of pointed blow-downs yield invariant mean linearizations]\label{lem:pointed-realization}
Every element of $\mathscr D^1_{\mathrm{pt}}(\Psi)$ belongs to $\mathcal A(\Psi)$.
\end{lemma}

\begin{proof}
Let $c_j \in N$ and $r_j \to \infty$ be sequences along which $F_j := \Psi_{c_j,r_j} = \sigma_{r_j}^M \circ \Psi^{c_j} \circ (\sigma_{r_j}^N)^{-1}$ converge locally uniformly to $F \in \mathscr B_{\mathrm{pt}}(\Psi)$, and let $p \in N_\infty$ be a Pansu differentiability point of $F$. Write $D^1 := D_p^1 F$. Observe that $\|\Phi_n^{c_j}\|_\infty = \|\Phi_n\|_\infty \le \Lambda|n|_{S_N} + A$ for all $j$ and $n\in N$. 

We retrace the proof of Lemma \ref{lem:mesoscopic-approximation}, applying it to the sequence $\Psi^{c_j}$. Choose $x_j \in N$ such that $\sigma_{r_j}^N(x_j) \to p$. For $n \in N$, fix $R \ge \max\{1, 2|n|_{S_N}\}$ and $t \in (0, 1]$. Let $k_j = \lceil t r_j \rceil$, $\rho_j = \lceil R t r_j \rceil$ and $\Omega_j' = x_j B_{\rho_j}$. Exactly as in the proof of \eqref{eq:telescoping-error}, we have
$$
 \limsup_j \left\| \frac{1}{|\Omega_j'|} \sum_{x \in \Omega_j'} \Phi_n^{c_j}(x) - \frac{1}{|\Omega_j'|} \sum_{x \in \Omega_j'} \frac{1}{k_j} \Phi_{n^{k_j}}^{c_j}(x) \right\| 
 \le \frac{C_{Q_N} |n|_{S_N} \|\Phi_n\|_\infty}{R}.
$$
where $\|\cdot\|$ denotes the $\ell_{S_M}$ norm. Following the proof of \eqref{eq:main-window-estimate}, we get
$$
 \limsup_j \left\| \frac{1}{|\Omega_j'|}\sum_{x \in \Omega_j'} \Phi_n^{c_j}(x) - D^1 \pi_N(n) \right\| 
 \le \frac{C_{Q_N} |n|_{S_N} \|\Phi_n\|_\infty}{R} + 2\eta(t,4R),
$$
where $\eta(t,R)$ is as in \eqref{eq:Pansu-differentiability-uniform}.

Fix finite subsets $E,F\subseteq N$, $\varepsilon>0$ and $R_0\ge 1$. Then, as in the last part of the proof of Lemma \ref{lem:mesoscopic-approximation}, there is $J_0\ge 1$ such that for every $j\ge J_0$, there are $a_j(=x_j)\in N$ and integers $\rho_j\ge R_0$ so that with $\Omega_j'=a_jB_{\rho_j}$ such that 
\begin{align*}
\max_{n\in E}\left\|\frac1{|\Omega_j'|}\sum_{x\in\Omega_j'}\Phi_n^{c_j}(x)-D^1\pi_N(n)\right\|&<\varepsilon, \quad\max_{g\in F}\frac{|\Omega_j' g\mathbin\triangle\Omega_j'|}{|\Omega_j'|}<\varepsilon,\\
\frac{\rho_j}{r_j}<\varepsilon, \qquad d_{N,\infty}(\sigma_{r_j}^N(a_j),p)<\varepsilon.
\end{align*}

We now execute the diagonal selection as in Theorem \ref{thm:translated-ball-realization}. Enumerate $N = \{n_1, n_2, \ldots\}$ and set $E_k = F_k=\{n_1, \ldots, n_k\}, \varepsilon=k^{-1}$ and $R_0=k$ Following the proof of Theorem $\ref{thm:translated-ball-realization}$, there exist $b_j\in N$, integers $\rho'_j$ with $\Omega''_j=b_jB_{\rho'_j}$ such that for every $n,g\in N$,
\begin{equation}\label{eq:shifted-diagonal-convergence}
 \frac{1}{|\Omega_j''|} \sum_{x \in \Omega_j''} \Phi_n^{c_j}(x) \longrightarrow D^1 \pi_N(n), 
 \qquad \frac{|\Omega_j'' g \mathbin\triangle \Omega_j''|}{|\Omega_j''|} \longrightarrow 0.
\end{equation}
Now,
$$
 \frac{1}{|\Omega_j''|} \sum_{x \in b_j B_{\rho'_j}} \Phi_n^{c_j}(x) 
 = \frac{1}{|B_{\rho'_j}|} \sum_{x \in b_j B_{\rho'_j}} \Phi_n(c_j x) 
 = \frac{1}{|c_j b_j B_{\rho'_j}|} \sum_{w \in c_j b_j B_{\rho'_j}} \Phi_n(w).
$$
Let $\Omega_j = c_j b_j B_{\rho'_j}$, then $(\Omega_j)$ form a F\o lner sequence. To complete the proof, we now trace the proof of Corollary \ref{cor:mean-realization}. Let $A_j$ be the positive linear functional on $\ell^\infty(N)$ defined by averaging a bounded function on $\Omega_j$, and $\mathsf A$ be the limit of a weak-* convergent subnet $(A_{j(\alpha)})$ of $(A_j)$. Note that $\mathsf A$ is actually a right-invariant mean on $N$ since $(\Omega_j)$ is F\o lner. Observe that
$$
 L_{\mathsf A} \pi_N(n) 
 = \mathsf A[\Phi_n] 
 = \lim_\alpha \frac{1}{|\Omega_{j(\alpha)}|} \sum_{w \in \Omega_{j(\alpha)}} \Phi_n(w) 
 = D^1 \pi_N(n).
$$
Therefore, $D_p^1 F = L_{\mathsf A} \in \mathcal A(\Psi)$.
\end{proof}

\begin{lemma}[Comparing $\mathscr D^1_{\mathrm{pt}}(\Psi)$ and $\mathscr G_{\mathrm{pt}}(\Psi)$]\label{lem:retangent-pointed}
One has
$$
 \mathscr D^1_{\mathrm{pt}}(\Psi)=\mathscr G_{\mathrm{pt}}(\Psi).
$$
Moreover, $\mathscr G_{\mathrm{pt}}(\Psi)$ is compact.
\end{lemma}

\begin{proof}
Let $c_j \in N$ and $r_j \to \infty$ be sequences along which $F_j := \sigma_{r_j}^M \circ \Psi^{c_j} \circ (\sigma_{r_j}^N)^{-1}$ converge locally uniformly to $F \in \mathscr B_{\mathrm{pt}}(\Psi)$ as in Definition \ref{def:pointed-blowdown}. Suppose that $p \in N_\infty$ is a Pansu differentiability point of $F$ with first-layer differential $D_p^1F = L$. By definition of Pansu differentiability, the map
$$
 T_{p,t}F(z) := \delta_{1/t}\bigl(F(p)^{-1}F(p\delta_tz)\bigr)
$$
converges uniformly to the graded homomorphism $\widehat L = D_pF$ as $t \to 0^+$ on every compact subset of $N_\infty$.

We will construct a discrete analogue of this map. Choose a decreasing sequence $(t_k)$ converging to $0$ such that $T_{p,t_k}F$ is within $1/k$ of $\widehat L$ on $B_{N,\infty}(e, k)$. Write $\operatorname{inv}_r^M$ for inversion with respect to the transported product $*_r^M$. For fixed $k,j$, and an element $a \in N$, set $q = \sigma_{r_j}^N(a)$ and define 
\begin{align*}
 \mathcal T_{j,a,k}(z) := \delta_{1/t_k}\Bigl(& \operatorname{inv}_{r_j}^M(F_j(q))*_{r_j}^M F_j\bigl(q*_{r_j}^N\delta_{t_k}z\bigr) \Bigr) \quad(z\in\sigma_{t_k r_j}^N(N)).
\end{align*}
Let $z = \sigma_{t_k r_j}^N(x)$, $x\in N$. Then $\delta_{t_k} z = \sigma_{r_j}^N(x)$, which gives
$$
 q *_{r_j}^N \delta_{t_k} z = \sigma_{r_j}^N(a) *_{r_j}^N \sigma_{r_j}^N(x) = \sigma_{r_j}^N(ax).
$$
Therefore, $F_j(q *_{r_j}^N \delta_{t_k} z)=\sigma_{r_j}^M\bigl(\Psi(c_j)^{-1}\Psi(c_j ax)\bigr)$. Now,
$$
 \operatorname{inv}_{r_j}^M(F_j(q)) = \operatorname{inv}_{r_j}^M \sigma_{r_j}^M\bigl(\Psi(c_j)^{-1}\Psi(c_j a)\bigr) = \sigma_{r_j}^M\bigl(\Psi(c_j a)^{-1}\Psi(c_j)\bigr).
$$
Since $\sigma_{r_j}^M\bigl(\Psi(c_j a)^{-1}\Psi(c_j)\bigr) *_{r_j}^M \sigma_{r_j}^M\bigl(\Psi(c_j)^{-1}\Psi(c_j ax)\bigr) = \sigma_{r_j}^M\bigl(\Psi(c_j a)^{-1}\Psi(c_j ax)\bigr)$, we have
$$
 \mathcal T_{j,a,k}(\sigma_{t_k r_j}^N(x)) = \sigma_{t_kr_j}^M\bigl( \Psi(c_ja)^{-1}\Psi(c_jax) \bigr).
$$
Because $t_k$ is fixed as $j \to \infty$, \eqref{eq:cone-density} guarantees we can choose $a_k \in N$ such that $q_k = \sigma_{r_{j(k)}}^N(a_k)$ converges to $p$. Because the original maps $F_j$, along with the transported products and inverses, converge uniformly on compact sets to their continuous counterparts, we can choose an increasing sequence of indices $j(k) > j(k-1)$ ensuring that $t_kr_{j(k)} \ge k$ and
\begin{equation}\label{eq:direct-retangent-comparison}
 \sup_{\substack{z\in\sigma_{t_kr_{j(k)}}^N(N)\\d_{N,\infty}(e,z)\le k}}
 d_{M,\infty}\bigl( \mathcal T_{j(k),a_k,k}(z), T_{p,t_k}F(z) \bigr) < \frac1k.
\end{equation}
By \eqref{eq:direct-retangent-comparison}, $\mathcal T_{j(k),a_k,k}$ converges locally uniformly to $\widehat L$. This proves that $\widehat L$ is a pointed blow-down, and therefore $\mathscr D^1_{\mathrm{pt}}(\Psi) \subseteq \mathscr G_{\mathrm{pt}}(\Psi)$. The reverse inclusion follows trivially because the Pansu differential of the graded homomorphism $\widehat L$ at the identity is precisely $\widehat L$, giving its first layer $L$.

We now prove that $\mathscr G_{\mathrm{pt}}(\Psi)$ is compact. By Lemma \ref{lem:first-layer-biLip}, every $L \in \mathscr G_{\mathrm{pt}}(\Psi)$ satisfies the stable-norm bounds $\Lambda^{-1}\ell_{S_N}(v) \le \ell_{S_M}(Lv) \le \Lambda\ell_{S_N}(v)$. Thus, $\mathscr G_{\mathrm{pt}}(\Psi)$ is bounded. 

To show closedness, let $L_i \to L$ be a convergent sequence with $L_i \in \mathscr G_{\mathrm{pt}}(\Psi)$. Let us fix a basis for each layer of the Lie algebra consisting of elements formed by iterated brackets of first-layer basis vectors. The convergence $L_i \to L$ guarantees that the bracket relations pass to the limit, yielding a graded homomorphism $\widehat L$ whose first layer map is $L$. 

Finally, we need to show that $\widehat L \in \mathscr B_{\mathrm{pt}}(\Psi)$. Because $L_i \in \mathscr G_{\mathrm{pt}}(\Psi)$, each corresponding graded homomorphism $\widehat L_i$ belongs to $\mathscr B_{\mathrm{pt}}(\Psi)$. By Definition \ref{def:pointed-blowdown}, this means that for each fixed $i \ge 1$, there exists sequences $c_{i,j} \in N$ and $r_{i,j} \to \infty$ (as $j \to \infty$) such that the maps
$$
 H_{i,j} := \sigma_{r_{i,j}}^M \circ \Psi^{c_{i,j}} \circ (\sigma_{r_{i,j}}^N)^{-1}
$$
converge to $\widehat L_i$ as $j \to \infty$ locally uniformly.

For each $i$, we restrict our attention to $B_{N,\infty}(e, i)$. Because $H_{i,j} \to \widehat L_i$ uniformly on this bounded set as $j \to \infty$, we can choose a sufficiently large index $j(i)$ such that $r_{i,j(i)} \ge i$ and 
\begin{equation}\label{eq:diagonal-selection}
 \sup_{z \in \sigma_{r_{i,j(i)}}^N(N) \cap B_{N,\infty}(e, i)}
 d_{M,\infty}\bigl(H_{i,j(i)}(z), \widehat L_i(z)\bigr) < \frac{1}{i}.
\end{equation}

We define $\tilde c_i := c_{i, j(i)}$, $\tilde r_i := r_{i, j(i)}$, and $\tilde F_i := H_{i, j(i)}$. Since $\tilde r_i \ge i$, we have $\tilde r_i \to \infty$ as $i\to\infty$. We now verify that $\tilde F_i$ converges locally uniformly to $\widehat L$. Fix an arbitrary radius $R \ge 1$ and $\varepsilon > 0$. Because $\widehat L_i \to \widehat L$ locally uniformly, there exists an integer $I_1 \ge 1$ such that for all $i \ge I_1$, 
$$
 \sup_{z \in B_{N,\infty}(e, R)} d_{M,\infty}\bigl(\widehat L_i(z), \widehat L(z)\bigr) < \frac{\varepsilon}{2}.
$$
Choose an integer $I > \max\{R, I_1, 2/\varepsilon\}$. Consider any $i \ge I$ and any point $z \in \sigma_{\tilde r_i}^N(N) \cap B_{N,\infty}(e, R)$. Because $i \ge I > R$, the point $z$ is contained within the larger ball $B_{N,\infty}(e, i)$, meaning the bound from \eqref{eq:diagonal-selection} applies. By the triangle inequality,
\begin{align*}
 d_{M,\infty}\bigl(\tilde F_i(z), \widehat L(z)\bigr) 
 &\le d_{M,\infty}\bigl(\tilde F_i(z), \widehat L_i(z)\bigr) + d_{M,\infty}\bigl(\widehat L_i(z), \widehat L(z)\bigr) \\
 &< \frac{1}{i} + \frac{\varepsilon}{2}.
\end{align*}
Since $i > 2/\varepsilon$, we have $1/i < \varepsilon/2$. Substituting this yields
$$
 d_{M,\infty}\bigl(\tilde F_i(z), \widehat L(z)\bigr) < \frac{\varepsilon}{2} + \frac{\varepsilon}{2} = \varepsilon.
$$
Since this bound holds uniformly for all $z\in \sigma_{\tilde r_i}^N(N) \cap B_{N,\infty}(e, R)$, $\tilde F_i$ converges locally uniformly to $\widehat L$. 

Therefore, $\widehat L \in \mathscr B_{\mathrm{pt}}(\Psi)$ and $L \in \mathscr G_{\mathrm{pt}}(\Psi)$, proving that $\mathscr G_{\mathrm{pt}}(\Psi)$ is closed. Since it was already shown to be bounded, $\mathscr G_{\mathrm{pt}}(\Psi)$ is compact.
\end{proof}

Fix elements $s_1,\ldots,s_d\in N$ for which $v_i:=\pi_N(s_i)$ form a basis of $V_N$, and define the bounded map
\begin{equation}\label{eq:def-J}
 J:N\longrightarrow\Hom(V_N,V_M),
 \qquad J(x)v_i:=\Phi_{s_i}(x).
\end{equation}

Let $\mu_N$ be the Haar measure on $N_\infty$ and put $\mathbb B:=B_{N,\infty}(e,1)$.  

\begin{lemma}\label{lem:discrete-weak-derivatives}
Suppose that $c_j\in N$, $r_j\to\infty$, and the corresponding maps $\Psi_{c_j,r_j}$ in \eqref{eq:rescaled-normalized-translate} satisfy \eqref{eq:pointed-uniform-conv} for $F\in\mathscr B_{\mathrm{pt}}(\Psi)$. Then
\begin{equation}\label{eq:weak-gradient-ball}
 \frac1{|B_{r_j}|}\sum_{x\in B_{r_j}}J(c_jx)
 \longrightarrow
 \frac1{\mu_N(\mathbb B)}
 \int_{\mathbb B}D_p^1F\,d\mu_N(p).
\end{equation}
\end{lemma}

\begin{proof}
Fix a basis vector $v=v_i$ and write $s=s_i$. Put $Q:=Q_N$ and
$$
z_{j,x}:=\sigma_{r_j}^N(x),\qquad
b_j(x):=\Phi_s^{c_j}(x)=\Phi_s(c_jx),\qquad
u_j(z_{j,x}):=\frac1{r_j}\pi_M(\Psi^{c_j}(x)).
$$
The displacement bound gives
$$
\|b_j(x)\|\le M_s:=\Lambda|s|_{S_N}+A
\qquad(j\ge1,\ x\in N).
$$
Since $\pi_M$ is a homomorphism,
$$
b_j(x)=r_j\bigl(u_j(z_{j,xs})-u_j(z_{j,x})\bigr).
$$
Also, $u_j=\pi_M^\infty\circ\Psi_{c_j,r_j}$ on $D_{r_j}^N$. The first-layer projection is $1$-Lipschitz, so \eqref{eq:pointed-uniform-conv} implies
$$
\sup_{z\in D_{r_j}^N\cap\overline B_{N,\infty}(e,R)}
\|u_j(z)-u(z)\|\longrightarrow0,
\qquad
u:=\pi_M^\infty\circ F,
$$
for every fixed $R$. The map $u$ is $\Lambda$-Lipschitz.

We first establish the lattice Riemann-sum formula. Choose Haar measure $\nu$ on $\widetilde N$ so that $(\iota_N)_*\nu=\mu_N$. This normalization is possible because Haar measures have constant densities in exponential coordinates and $\iota_N$ is linear in those coordinates. Choose a relatively compact measurable fundamental domain $K$ for the left action of $N$ on $\widetilde N$, so that the sets $xK$, $x\in N$, partition $\widetilde N$ up to null sets. Put $v_0:=\nu(K)$ and $\kappa_N:=v_0^{-1}$.

The cells
$$
E_{r,x}:=\sigma_r^N(xK),\qquad x\in N,
$$
partition $N_\infty$ up to null sets. Since the dilation $\delta_{1/r}$ has Jacobian $r^{-Q}$,
$$
\mu_N(E_{r,x})=v_0r^{-Q}.
$$
Write $z_{r,x}:=\sigma_r^N(x)$. For $k\in K$,
$$
\sigma_r^N(xk)=z_{r,x}*_r^N\sigma_r^N(k).
$$
The points $\sigma_r^N(k)$ tend uniformly to $e$ for $k\in K$, and the transported products $*_r^N$ converge locally uniformly to the product of $N_\infty$. Consequently, for every compact set $C$,
$$
\sup_{\substack{x\in N\\z_{r,x}\in C}}
\sup_{z\in E_{r,x}}d_{N,\infty}(z,z_{r,x})
\longrightarrow0.
$$
If $E_{r,x}\cap C\neq \emptyset$ for some $r,x$, choose $k\in K$ such that $z=\sigma_r^N(xk)\in C$. Then
$$
z_{r,x}=z*_r^N\sigma_r^N(k^{-1}).
$$
Thus the anchors of all cells $E_{r,x}$ meeting $C$ remain in a fixed compact set. It follows that all cells meeting $C$, or having their anchors in $C$, lie together with their anchors in one compact set $C'$ for sufficiently large $r$. Every point of such a cell is within $\varepsilon_r$ of its anchor, where $\varepsilon_r\to0$.

Let $h\in C_c(N_\infty)$ and take $C=\operatorname{supp}h$. Then,
$$
v_0r^{-Q}\sum_{x\in N}h(z_{r,x})-\int_{N_\infty} h\,d\mu_N
=
\sum_{x\in N}\int_{E_{r,x}}
\bigl(h(z_{r,x})-h(z)\bigr)\,d\mu_N(z).
$$
If $\omega_h$ is a modulus of continuity for $h$ on $C'$, the norm of the right-hand side is at most $\mu_N(C')\omega_h(\varepsilon_r)$, which tends to zero. Therefore
\begin{equation}\label{eq:lattice-riemann-sum}
r^{-Q}\sum_{x\in N}h(\sigma_r^N(x))
\longrightarrow\kappa_N\int_{N_\infty}h\,d\mu_N.
\end{equation}
In particular, for every compact $C\subset N_\infty$,
\begin{equation}\label{eq:lattice-counting-bound}
r_j^{-Q}\#\{x\in N:z_{r_j,x}\in C\}=O_C(1).
\end{equation}
Indeed, apply \eqref{eq:lattice-riemann-sum} to a nonnegative continuous compactly supported function that is at least one on $C$.

Let $\eta\in C_c^\infty(N_\infty)$. Using the identity for $b_j$ and then changing variables $x\mapsto xs$ in the first sum gives
\begin{align}
r_j^{-Q}\sum_x\eta(z_{j,x})b_j(x)
&=r_j^{1-Q}\sum_x\eta(z_{j,x})
\bigl(u_j(z_{j,xs})-u_j(z_{j,x})\bigr)\notag\\
&=r_j^{-Q}\sum_x
r_j\bigl(\eta(z_{j,xs^{-1}})-\eta(z_{j,x})\bigr)u_j(z_{j,x}).
\label{eq:discrete-summation-by-parts}
\end{align}
Note that these sums are finite. We compute the finite difference in \eqref{eq:discrete-summation-by-parts}. Let
$$
X_v\eta(z):=\left.\frac{d}{dt}\right|_{t=0}
\eta(z\exp_\infty(tv)).
$$
Identify the Lie algebra of $\widetilde N$ with $\mathfrak n_\infty$ as in Section \ref{sec:preliminaries}. For $Z_a\in V_a^N$ and $W_b\in V_b^N$, let $\beta_{ab}^k(Z_a,W_b)$ denote the component in $V_k^N$ of their original bracket under this identification. For $r>0$, we define
$$
[Z_a,W_b]_r:=\sum_{k\ge a+b}r^{a+b-k}\beta_{ab}^k(Z_a,W_b).
$$
Let $P_r(Z,W)$ be the Baker--Campbell--Hausdorff polynomial for $[\cdot,\cdot]_r$, and let $P_\infty(Z,W)$ be the corresponding polynomial for the graded bracket on $\mathfrak n_\infty$. Collecting like powers of $r^{-1}$ gives the identity
$$
P_r(Z,W)=P_\infty(Z,W)+\sum_{h=1}^{H}r^{-h}P_h(Z,W),
$$
where $H$ is finite and the maps $P_h$ are independent of $r$ and described below.

Since $P_r(Z,0)=Z=P_\infty(Z,0)$ for every $r>0$,
$$
\sum_{h=1}^{H}r^{-h}P_h(Z,0)=0.
$$
For fixed $Z$, this is a polynomial identity in $r^{-1}$, so each coefficient vanishes: $P_h(Z,0)=0$.

Fix a Euclidean norm $\|\cdot\|$ on $\mathfrak n_\infty$, an orthonormal basis $E_1,\dots E_d$ of $\mathfrak n_\infty$ and a compact set $\mathcal K\subset\mathfrak n_\infty$. Write
$$
Z=\sum_{i=1}^d z_iE_i,\qquad W=\sum_{i=1}^d w_iE_i,\qquad
P_h(Z,W)=\sum_{i=1}^d p_{h,i}(z,w)E_i,
$$
where each $p_{h,i}$ is a polynomial in the $2d$ real variables $(z,w)$. Since $P_h(Z,0)=0$, we have $p_{h,i}(z,0)=0$ for all $i$. Define 
$$
D_2P_h(Z,W):=\left(\frac{\partial p_{h,i}}{\partial w_j}(z,w)\right)_{i,j=1}^d.
$$ 
Its entries are polynomials, so continuity and compactness give
$$
C_{h,\mathcal K}:=
\sup_{\substack{Z\in\mathcal K\\ \|U\|\le1}}
\|D_2P_h(Z,U)\|_{\infty}<\infty,
$$
where $\|(a)_{i,j}\|_\infty=\max_{i,j}|a_{i,j}|$. For fixed $Z$ and $W$, the ordinary chain rule gives
$$
\frac{d}{dt}p_{h,i}(z, tw)
=\sum_{j=1}^d\frac{\partial p_{h,i}}{\partial w_j}(z,tw)w_j.
$$
Applying the fundamental theorem of calculus gives
$$
p_{h,i}(z,w)=p_{h,i}(z,w)-p_{h,i}(z,0)=\sum_{j=1}^d\left(\int_0^1\frac{\partial p_{h,i}}{\partial w_j}(z,tw)\,dt\right) w_j.
$$
Consequently, for $Z\in\mathcal K$ and $W\in\mathfrak n_\infty$ with $\|W\|\le 1$, we have
$$
\|P_h(Z,W)\|\le dC_{h,\mathcal K}\sum_{j=1}^d |w_j| \le d^2C_{h,\mathcal K}\|W\|.
$$
Since $r^{-h}\le r^{-1}$ for $r\ge1$ and $h\ge1$, 
\begin{align*}
\|P_r(Z,W)-P_\infty(Z,W)\|
&\le\sum_{h=1}^{H}r^{-h}\|P_h(Z,W)\|\\
&\le C_{\mathcal K,d}r^{-1}\|W\|.
\end{align*}
The constant $C_{\mathcal K,d}$ is independent of $r$, $Z\in\mathcal K$, and $W$ with $\|W\|\le1$. In particular, an argument $W$ of size $O(r^{-1})$ gives an error of size $O(r^{-2})$.

Write $\log_{\widetilde N}s=\sum_{k\ge1}Y_k$, where $Y_k\in W_k$ and $q_1(Y_1)=v$. Since $\log_{\widetilde N}(s^{-1})=-\log_{\widetilde N}s$, the definition of $\sigma_r^N$ gives
$$
U_r:=\log_\infty\bigl(\sigma_r^N(s^{-1})\bigr)
=-\sum_{k\ge1}r^{-k}q_k(Y_k).
$$
The sum is finite and $s$ is fixed, so
$$
\|U_r+v/r\|=O(r^{-2}),
\qquad
\|U_r\|=O(r^{-1}).
$$

Fix a compact set $\mathcal K\subset\mathfrak n_\infty$. The derivative $D_2P_\infty$ is bounded on $\mathcal K$ times the closed unit ball. For sufficiently large $r$, both $U_r$ and $-v/r$ lie in that ball. The preceding estimate and the fundamental theorem of calculus in the second variable therefore give
\begin{align*}
\|P_r(Z,U_r)-P_\infty(Z,-v/r)\|
&\le\|P_r(Z,U_r)-P_\infty(Z,U_r)\|+\|P_\infty(Z,U_r)-P_\infty(Z,-v/r)\|\\
&\le C_{\mathcal K,d}r^{-1}\|U_r\|
+C'_{\mathcal K}\|U_r+v/r\|\\
&=O(r^{-2}),
\end{align*}
uniformly for $Z\in\mathcal K$.

For $z=\sigma_r^N(x)=\exp_\infty Z$, the transported product gives the identities
$$
\sigma_r^N(xs^{-1})=\exp_\infty(P_r(Z,U_r)),
\qquad
z\exp_\infty(-v/r)=\exp_\infty(P_\infty(Z,-v/r)).
$$
Put $f:=\eta\circ\exp_\infty$. This is a smooth compactly supported function on $\mathfrak n_\infty$, so its derivative is bounded. The ordinary mean value estimate gives
\begin{align*}
\bigl|\eta(\sigma_r^N(xs^{-1}))
-\eta(z\exp_\infty(-v/r))\bigr|
&=\bigl|f(P_r(Z,U_r))-f(P_\infty(Z,-v/r))\bigr|\\
&\le\|Df\|_\infty
\|P_r(Z,U_r)-P_\infty(Z,-v/r)\|\\
&=O(r^{-2}),
\end{align*}
uniformly when $z$ lies in a fixed compact set.

To expand the remaining term, define $\varphi_z(t):=\eta(z\exp_\infty(tv))$. By the definition of $X_v$,
$$
\varphi_z'(t)=(X_v\eta)(z\exp_\infty(tv)),
\qquad
\varphi_z''(t)=(X_v^2\eta)(z\exp_\infty(tv)).
$$
Taylor's formula with integral remainder gives
$$
\eta(z\exp_\infty(-v/r))
=\eta(z)-\frac1rX_v\eta(z)
+\frac1{r^2}\int_0^1(1-\theta)
(X_v^2\eta)(z\exp_\infty(-\theta v/r))\,d\theta.
$$
Since $X_v^2\eta$ is bounded, the remainder is $O(r^{-2})$. Combining the two estimates and multiplying by $r_j$ yields
$$
\Delta_j(x):=
r_j\bigl(\eta(z_{j,xs^{-1}})-\eta(z_{j,x})\bigr)
=-X_v\eta(z_{j,x})+O(r_j^{-1}),
$$
uniformly when $z_{j,x}$ lies in a fixed compact set.

We now justify the passage to the limit in \eqref{eq:discrete-summation-by-parts}. Put $S:=\operatorname{supp}\eta$. If $\Delta_j(x)\ne0$, then either $z_{j,x}\in S$ or $z_{j,xs^{-1}}\in S$. In the latter case,
$$
z_{j,x}=z_{j,xs^{-1}}*_{r_j}^N\sigma_{r_j}^N(s).
$$
Since $\sigma_{r_j}^N(s)\to e$ and $*^N_{r_j}$ converges uniformly on compact sets to the multiplication of $N_{\infty}$, we can choose a compact set $\mathcal C$ containing $S$ so that the points $z_{j,x}$ lie in $\mathcal C$ for sufficiently large $j$. Since $\operatorname{supp}(X_v\eta)\subseteq S$, both $\Delta_j(x)$ and $X_v\eta(z_{j,x})$ vanish whenever $z_{j,x}\notin\mathcal C$.

Write $\xi:=X_v\eta$ and
$$
e_j(x):=\Delta_j(x)+\xi(z_{j,x}).
$$
The uniform expansion gives $|e_j(x)|\le C_\eta/r_j$ for $z_{j,x}\in\mathcal C$, while $e_j(x)=0$ outside $\mathcal C$. Moreover,
$$
\delta_j:=
\sup_{\substack{x\in N\\z_{j,x}\in\mathcal C}}
\|u_j(z_{j,x})-u(z_{j,x})\|
\longrightarrow0.
$$
Since $u$ is continuous, this also gives a constant $M$ with $\|u_j(z_{j,x})\|\le M$ on $\mathcal C$ for all sufficiently large $j$. \eqref{eq:lattice-counting-bound} supplies a constant $B$ such that
$$
r_j^{-Q}\#\{x\in N:z_{j,x}\in\mathcal C\}\le B.
$$
By \eqref{eq:discrete-summation-by-parts},
\begin{align*}
r_j^{-Q}\sum_x\eta(z_{j,x})b_j(x)
&=r_j^{-Q}\sum_x\Delta_j(x)u_j(z_{j,x})\\
&=-r_j^{-Q}\sum_x\xi(z_{j,x})u_j(z_{j,x})
+r_j^{-Q}\sum_xe_j(x)u_j(z_{j,x}).
\end{align*}
Consequently,
$$
\left\|
r_j^{-Q}\sum_x\eta(z_{j,x})b_j(x)
+r_j^{-Q}\sum_x\xi(z_{j,x})u(z_{j,x})
\right\|
\le B\left(\frac{C_\eta M}{r_j}
+\|\xi\|_\infty\delta_j\right)
\longrightarrow0.
$$
Finally, $\xi u$ is continuous and compactly supported. Applying \eqref{eq:lattice-riemann-sum} to its scalar coordinates gives
$$
r_j^{-Q}\sum_x\xi(z_{j,x})u(z_{j,x})
\longrightarrow
\kappa_N\int_{N_\infty}\xi u\,d\mu_N.
$$
Thus
$$
\lim_j r_j^{-Q}\sum_x\eta(z_{j,x})b_j(x)
=-\kappa_N\int_{N_\infty}(X_v\eta)u\,d\mu_N.
$$

We identify the last integral directly. At almost every $p$, Pansu differentiability of $F$ and projection to the first layer give
$$
\lim_{t\to 0}
\frac{u(p\exp_\infty(tv))-u(p)}t
=D_p^1F(v)=:g_v(p).
$$
Moreover,
$$
\ell_{S_M}\left(\frac{u(p\exp_\infty(tv))-u(p)}t\right)
\le\Lambda\ell_{S_N}(v)=:G_v.
$$
Thus $\|g_v(p)\|\le G_v$ for almost every $p$. The group $N_\infty$ is nilpotent and hence unimodular, so Haar measure is invariant under right translation by $\exp_\infty(tv)$. Changing variables in the first term below gives
\begin{align*}
&\int_{N_\infty}\eta(p)\frac{u(p\exp_\infty(tv))-u(p)}t\,d\mu_N(p)=\int_{N_\infty}\frac{\eta(p\exp_\infty(-tv))-\eta(p)}{t}u(p)\,d\mu_N(p).
\end{align*}
On the left, dominated convergence applies by the uniform difference-quotient bound. On the right, the test-function quotients are supported in one compact set and converge uniformly to $-X_v\eta$. Therefore
$$
\int\eta(p)g_v(p)\,d\mu_N(p)
=-\int(X_v\eta)u\,d\mu_N.
$$
We have proved
$$
r_j^{-Q}\sum_x\eta(z_{j,x})b_j(x)
\longrightarrow
\kappa_N\int\eta(p)g_v(p)\,d\mu_N(p)
\qquad(\eta\in C_c^\infty(N_\infty)).
$$

It remains to pass from smooth functions to the word balls. Put $\rho(p):=d_{N,\infty}(e,p)$ and $b:=\mu_N(\mathbb B)$. Homogeneity gives the scaling of closed-ball volumes, and
$$
B_{N,\infty}(e,t)
=\bigcup_{n:\,1/n<t}\overline B_{N,\infty}(e,t-1/n).
$$
Continuity from below therefore shows that every sphere of positive radius has measure zero and that
$$
\mu_N(\{\rho\le t\})=\mu_N(\{\rho<t\})=t^Qb.
$$
For $0<\delta<1/4$, put
$$
A_\delta:=\{p:1-\delta\le\rho(p)\le1+\delta\}.
$$
Then
$$
\mu_N(A_\delta)
=b\bigl((1+\delta)^Q-(1-\delta)^Q\bigr)\longrightarrow0
\qquad(\delta\downarrow0).
$$
Choose smooth compactly supported functions $\eta_\delta,\zeta_\delta$ with values in $[0,1]$ such that
$$
\eta_\delta=1\ \text{on }\{\rho\le1-\delta\},
\qquad
\operatorname{supp}\eta_\delta\subset\{\rho<1+\delta\},
$$
and
$$
\zeta_\delta=1\ \text{on }A_\delta,
\qquad
\operatorname{supp}\zeta_\delta
\subset\{1-2\delta<\rho<1+2\delta\}.
$$
These functions exist because the prescribed compact sets lie in the indicated open sets. We have
$$
\bigl|\mathbf1_{\{\rho\le1\}}-\eta_\delta\bigr|
\le\mathbf1_{A_\delta}\le\zeta_\delta.
$$
Applying \eqref{eq:lattice-riemann-sum} to $\zeta_\delta$ gives
$$
\limsup_j r_j^{-Q}\#\{x:z_{j,x}\in A_\delta\}
\le\kappa_N\int\zeta_\delta\,d\mu_N
\le\kappa_N\mu_N(A_{2\delta}).
$$

Let $E_j:=\{x\in N:\rho(z_{j,x})\le1\}$. The bounds on $b_j$ and $g_v$ give
\begin{align*}
&\left\|r_j^{-Q}\sum_{x\in E_j}b_j(x)
-\kappa_N\int_{\mathbb B}g_v\,d\mu_N\right\|\\
&\quad\le
M_s r_j^{-Q}\#\{x:z_{j,x}\in A_\delta\}\\
&\qquad+
\left\|r_j^{-Q}\sum_x\eta_\delta(z_{j,x})b_j(x)
-\kappa_N\int\eta_\delta g_v\,d\mu_N\right\|
+\kappa_NG_v\mu_N(A_\delta).
\end{align*}
The middle term tends to zero by the smooth-test convergence. Hence
$$
\limsup_j
\left\|r_j^{-Q}\sum_{x\in E_j}b_j(x)
-\kappa_N\int_{\mathbb B}g_v\,d\mu_N\right\|
\le
\kappa_NM_s\mu_N(A_{2\delta})
+\kappa_NG_v\mu_N(A_\delta).
$$
Letting $\delta\downarrow0$ proves
$$
r_j^{-Q}\sum_{x\in E_j}b_j(x)
\longrightarrow\kappa_N\int_{\mathbb B}g_v\,d\mu_N.
$$
The same estimate with $b_j\equiv1$ and $g_v\equiv1$, using \eqref{eq:lattice-riemann-sum} for its middle term, gives
$$
r_j^{-Q}|E_j|\longrightarrow\kappa_Nb.
$$

For $x\in E_j$, \eqref{eq:cone-bounded-lift} gives $|x|_{S_N}\le C_1r_j$. Thus every point of $E_j\cup B_{r_j}$ satisfies $|x|_{S_N}\le Cr_j$, where $C:=\max\{1,C_1\}$. By \eqref{eq:coordinate-pansu-convergence},
$$
\omega_j:=
\sup_{|x|_{S_N}\le Cr_j}
\left|\frac{|x|_{S_N}}{r_j}-\rho(z_{j,x})\right|
\longrightarrow0.
$$
If $x\in B_{r_j}\setminus E_j$, then $1<\rho(z_{j,x})\le1+\omega_j$. If $x\in E_j\setminus B_{r_j}$, then $1-\omega_j<\rho(z_{j,x})\le1$. Consequently, for each fixed $\delta>0$ and all sufficiently large $j$,
$$
E_j\mathbin\triangle B_{r_j}
\subseteq\{x:z_{j,x}\in A_\delta\}.
$$
The annular estimate therefore gives
$$
\limsup_j r_j^{-Q}|E_j\mathbin\triangle B_{r_j}|
\le\kappa_N\mu_N(A_{2\delta}).
$$
Letting $\delta\downarrow0$ shows that the left hand side of above tends to zero. Since
$$
r_j^{-Q}
\left\|\sum_{x\in B_{r_j}}b_j(x)-\sum_{x\in E_j}b_j(x)\right\|
\le M_s r_j^{-Q}|E_j\mathbin\triangle B_{r_j}|
$$
and
$$
r_j^{-Q}\bigl||B_{r_j}|-|E_j|\bigr|
\le r_j^{-Q}|E_j\mathbin\triangle B_{r_j}|,
$$
we obtain
$$
r_j^{-Q}\sum_{x\in B_{r_j}}b_j(x)
\longrightarrow\kappa_N\int_{\mathbb B}g_v\,d\mu_N,
\qquad
r_j^{-Q}|B_{r_j}|\longrightarrow\kappa_Nb>0.
$$
Dividing these limits yields
$$
\frac1{|B_{r_j}|}\sum_{x\in B_{r_j}}\Phi_s(c_jx)
\longrightarrow
\frac1{\mu_N(\mathbb B)}
\int_{\mathbb B}D_p^1F(v)\,d\mu_N(p).
$$
Since $J(c_jx)v_i=\Phi_{s_i}(c_jx)$ and there are finitely many basis vectors, this proves \eqref{eq:weak-gradient-ball}.
\end{proof}

\begin{lemma}\label{lem:translated-ball-support}
For $R\ge1$ and $c\in N$, put
$$
 J_R(c):=\frac1{|B_R|}\sum_{x\in B_R}J(cx).
$$
Then, for every $\beta\in\Hom(V_N,V_M)^*$,
\begin{equation}\label{eq:support-formula}
 \max_{L\in\mathcal A(\Psi)}\beta(L)
 =\lim_{R\to\infty}\sup_{c\in N}\beta(J_R(c)).
\end{equation}
\end{lemma}

\begin{proof}
For every right-invariant mean $\mathsf A$ on $N$, 
$$
\mathsf A[J]v_i=\mathsf A[J(\cdot)v_i]=\mathsf A[\Phi_{s_i}]=L_{\mathsf A}v_i.
$$
Since the $v_i$ form a basis of $V_N$, $\mathsf A[J]=L_{\mathsf A}$. Therefore, $\beta(L_{\mathsf A})=\mathsf A(\beta\circ J)$. For each fixed $x\in N$, define the bounded function $\xi_x:N\to \RR$ by 
$$
\xi_x(y)=\beta(J(yx)) \quad(y\in N).
$$
Then, $\xi_x=R_x(\beta\circ J)$, so $\mathsf A(\xi_x)=\mathsf A(\beta\circ J)$ for all $x\in N$. Hence, we get
$$
\beta(L_{\mathsf A})
=\mathsf A\left[\frac1{|B_R|}\sum_{x\in B_R}\xi_x\right]
=\mathsf A(\beta\circ J_R)
\le\sup_{c\in N}\beta(J_R(c)).
$$
Put $a_R:=\sup_{c\in N}\beta(J_R(c))$ and $h:=\max_{L\in\mathcal A(\Psi)}\beta(L)$. The preceding inequality gives $h\le a_R$ for every $R\ge1$, hence $h\le\liminf_{R\to\infty}a_R$.

Choose $R_j\to\infty$ such that $a_{R_j}\to\limsup_{R\to\infty}a_R$, and choose $c_j$ with $\beta(J_{R_j}(c_j))>a_{R_j}-1/j$. Let $S_j:=c_jB_{R_j}$ and 
$$
A_j(f):=\frac{1}{|S_j|}\sum_{x\in S_j}f(x), \quad (f\in\ell^\infty(N)).
$$
By Banach-Alaoglu theorem, there is a weak-* convergent subnet $(A_{j(\alpha)})$ of $(A_j)$ converging to a state $\mathsf A$ on $\ell^\infty(N)$. Since $(S_j)$ is a right F\o lner sequence, $\mathsf A$ is right invariant. For every basis vector $v_i$,
\begin{equation}\label{eq:folner-argument}
J_{R_{j(\alpha)}}(c_{j(\alpha)})v_i
=A_{j(\alpha)}[\Phi_{s_i}]
\longrightarrow\mathsf A[\Phi_{s_i}]
=L_{\mathsf A}v_i.
\end{equation}
Since $V_N$ and $V_M$ are finite dimensional, this proves
$$
J_{R_{j(\alpha)}}(c_{j(\alpha)})\longrightarrow L_{\mathsf A}.
$$
Moreover, $j(\alpha)\to\infty$, and $0\le a_{R_j}-\beta(J_{R_j}(c_j))<1/j.$ Consequently,
$$
\limsup_{R\to\infty}a_R=\beta(L_{\mathsf A})\le h.
$$
Hence, $\limsup_{R\to\infty}a_R\le h \le \liminf_{R\to\infty}a_R$, which proves \eqref{eq:support-formula}.

\end{proof}

\begin{theorem}[Invariant-mean linearizations are convex hulls of first layer differentials]
\label{thm:pointed-convexification}
For every quasi-isometry $\Psi:N\to M$ as above,
\begin{equation}\label{eq:exact-pointed-convexification}
 \mathcal A(\Psi)
 =\operatorname{conv}\mathscr D^1_{\mathrm{pt}}(\Psi)
 =\operatorname{conv}\mathscr G_{\mathrm{pt}}(\Psi).
\end{equation}
Moreover,
\begin{equation}\label{eq:extreme-pointed}
 \operatorname{Ext}\mathcal A(\Psi)
 \subseteq\mathscr G_{\mathrm{pt}}(\Psi).
\end{equation}
In particular, every extreme point of $\mathcal A(\Psi)$ extends to a graded $\Lambda$-bi-Lipschitz isomorphism and occurs as a homogeneous pointed blow-down of $\Psi$.
\end{theorem}

\begin{proof}
By Lemmas \ref{lem:pointed-realization} and \ref{lem:retangent-pointed}, $\mathscr G_{\mathrm{pt}}(\Psi)\subseteq\mathcal A(\Psi)$. Since $\mathcal A(\Psi)$ is convex,
$$
\operatorname{conv}\mathscr G_{\mathrm{pt}}(\Psi)\subseteq\mathcal A(\Psi).
$$
Let $L$ be an exposed point of $\mathcal A(\Psi)$, exposed by $\beta\in\Hom(V_N,V_M)^*$, and put $h=\max_{K\in\mathcal A(\Psi)}\beta(K)$. By Lemma \ref{lem:translated-ball-support}, choose $R_j\to\infty$ and $c_j\in N$ such that
$$
 \beta(J_{R_j}(c_j))\longrightarrow h.
$$
Every cluster point of $J_{R_j}(c_j)$ belongs to $\mathcal A(\Psi)$ by the argument used in the proof of Lemma \ref{lem:translated-ball-support} to obtain \eqref{eq:folner-argument}. Since $\beta$ exposes $L$, it follows, after passing to a subsequence, that
\begin{equation}\label{eq:exposed-ball-average}
J_{R_j}(c_j)\longrightarrow L.
\end{equation}
Pass to a further subsequence so that the translates $\Psi^{c_j}$ at scales $R_j$ converge to a pointed blow-down $F$. By Lemma \ref{lem:discrete-weak-derivatives},
$$
L=\frac1{\mu_N(\mathbb B)}\int_{\mathbb B}D_p^1F\,d\mu_N(p).
$$
For almost every $p$, Lemma \ref{lem:pointed-realization} gives $D_p^1F\in\mathcal A(\Psi)$, so $\beta(D_p^1F)\le h$. Since
$$
\frac1{\mu_N(\mathbb B)}\int_{\mathbb B}\beta(D_p^1F)\,d\mu_N(p)=\beta(L)=h,
$$ 
this forces $\beta(D_p^1F)=h$ for almost every $p\in\mathbb B$. As $L$ is an exposed point of $\mathcal A(\Psi)$ exposed by $\beta$, we must have
$$
D_p^1F=L\quad\text{for almost every }p\in\mathbb B.
$$
Choose a Pansu differentiability point $p\in\mathbb B$ with this property. By Lemma \ref{lem:retangent-pointed}, 
$$
L=D_p^1F\in\mathscr D^1_{\mathrm{pt}}(\Psi)=\mathscr G_{\mathrm{pt}}(\Psi).
$$ Thus every exposed point of $\mathcal A(\Psi)$ lies in $\mathscr G_{\mathrm{pt}}(\Psi)$. Now, exposed points are dense in the set of all extreme points of a compact convex set in finite dimensions. Since $\mathscr G_{\mathrm{pt}}(\Psi)$ is closed by Lemma \ref{lem:retangent-pointed}, this proves \eqref{eq:extreme-pointed}. By Krein-Milman theorem, $\mathcal A(\Psi)\subseteq\operatorname{conv}\mathscr G_{\mathrm{pt}}(\Psi)$, which completes the proof since $\mathscr D^1_{\mathrm{pt}}(\Psi)=\mathscr G_{\mathrm{pt}}(\Psi)$.
\end{proof}

\begin{corollary}[Singleton criterion]
\label{cor:uniform-singleton}
For a linear map $L:V_N\to V_M$, the following are equivalent.
\begin{enumerate}
\item $\mathcal A(\Psi)=\{L\}$.
\item Every pointed blow-down of $\Psi$ is the same graded map $\widehat L$.
\item With respect to any fixed norm $\|\cdot\|$ on $V_M$,
\begin{equation}\label{eq:uniform-sublinear-increments}
 \lim_{R\to\infty}
 \sup_{\substack{x,n\in N\\|n|_{S_N}\ge R}}
 \frac{\bigl\|\Phi_n(x)-L\pi_N(n)\bigr\|}{|n|_{S_N}}=0.
\end{equation}
\end{enumerate}
\end{corollary}

\begin{proof}
If $\mathcal A(\Psi)=\{L\}$, then Theorem \ref{thm:pointed-convexification} gives $\mathscr D^1_{\mathrm{pt}}(\Psi)=\{L\}$. Every pointed blow-down therefore has first-layer Pansu differential $L$ almost everywhere, and Lemma \ref{lem:constant-pansu-differential} shows that it equals $\widehat L$. Conversely, if every pointed blow-down is $\widehat L$, then $\mathscr G_{\mathrm{pt}}(\Psi)=\{L\}$ and \eqref{eq:exact-pointed-convexification} gives $\mathcal A(\Psi)=\{L\}$.

If \eqref{eq:uniform-sublinear-increments} fails while every pointed blow-down were $\widehat L$, there exists $x_j,n_j\in N$ with $|n_j|_{S_N}\to\infty$ as $j\to\infty$ such that
\begin{equation}\label{eq:singleton-criterion-contradiction}
\|\Phi_{n_j}(x_j)-L\pi_N(n_j)\|\ge\varepsilon|n_j|_{S_N}.
\end{equation}
Set $r_j:=|n_j|_{S_N}$, then \eqref{eq:coordinate-pansu-convergence} gives $d_{N,\infty}(e,\sigma_{r_j}^N(n_j))\to1$ as $j\to\infty$. Since $(N_\infty, d_{N,\infty})$ is a proper metric space, after passing to a common subsequence we may assume that
\begin{equation}\label{eq:intermediate-convergence}
 \sigma_{r_j}^N(n_j)\longrightarrow q
 \quad\text{and}\quad
 \Psi_{x_j,r_j}\longrightarrow F\in\mathscr B_{\mathrm{pt}}(\Psi)
\end{equation}
in the sense of \eqref{eq:pointed-uniform-conv}. Assumption (2) gives $F=\widehat L$. Since $\Psi_{x_j,r_j}(\sigma_{r_j}^N(n_j))=\sigma_{r_j}^M(\Psi(x_j)^{-1}\Psi(x_jn_j))$, using \eqref{eq:intermediate-convergence} and the continuity of $\pi_M^\infty$ we get
$$
 \frac1{r_j}\Phi_{n_j}(x_j)
 \longrightarrow\pi_M^\infty(\widehat L(q))
 =L\pi_N^\infty(q)
 =\lim_j\frac1{r_j}L\pi_N(n_j),
$$
contradicting \eqref{eq:singleton-criterion-contradiction}. Thus (2) implies (3).

Finally, assume (3). For $n\in N\setminus\{e_N\}$ and $k\ge1$, apply (3) to $n^k$. Then $|n^k|_{S_N}/k\to\ell_{S_N}(\pi_N(n))$ by \eqref{eq:power-stable-length}, and hence
$$
 \sup_{x\in N}\frac1k
 \|\Phi_{n^k}(x)-kL\pi_N(n)\|=\frac{|n^k|}k\sup_{x\in N}\frac{1}{|n^k|}
 \|\Phi_{n^k}(x)-L\pi_N(n^k)\|\longrightarrow0.
$$
For $k\in\NN$, define $E_k(x)=L\pi_N(n)-\frac{1}{k}\Phi_{n^k}(x)$, $x\in N$. Take a right-invariant mean $\mathsf A$ on $N$. Then, $L\pi_N(n)=\mathsf A[E_k]+L_{\mathsf A}\pi_N(n)$ for each $n\in N$, which gives $L\pi_N(n)=L_{\mathsf A}\pi_N(n)$ on taking $k\to\infty$. Since the image of $\pi_N$ spans $V_N$, we have $L_{\mathsf A}=L$ and $\mathcal A(\Psi)=\{L\}$.
\end{proof}

\begin{definition}[Bounded Abelian defect]
\label{def:bounded-abelian-defect}
A quasi-isometry $\Psi$ has \emph{bounded Abelian defect} if
$$
 \Delta_{\mathrm{ab}}(\Psi)
 :=\sup_{x,y\in N}
 \ell_{S_M}\!\left(
   \pi_M(\Psi(xy))-\pi_M(\Psi(x))-\pi_M(\Psi(y))
 \right)<\infty.
$$
\end{definition}

\begin{proposition}[Bounded-defect rigidity and functoriality]
\label{prop:bounded-defect-rigidity}
Suppose that $D:=\Delta_{\mathrm{ab}}(\Psi)<\infty$. There is a unique linear map $L_\Psi:V_N\to V_M$ such that $\pi_M\circ\Psi-L_\Psi\circ\pi_N$ is bounded. In fact,
\begin{equation}\label{eq:bounded-defect-linearization}
 \sup_{n\in N}
 \ell_{S_M}\!\left(\pi_M(\Psi(n))-L_\Psi\pi_N(n)\right)\le D.
\end{equation}
Moreover,
\begin{equation}\label{eq:bounded-defect-uniform-increments}
 \sup_{x,n\in N}
 \ell_{S_M}\!\left(\Phi_n(x)-L_\Psi\pi_N(n)\right)\le2D,
\end{equation}
$\mathcal A(\Psi)=\{L_\Psi\}$, and every pointed blow-down of $\Psi$ is the same graded bi-Lipschitz isomorphism $\widehat L_\Psi:N_\infty\to M_\infty$. In particular, $L_\Psi$ is a linear isomorphism.

The assignment $\Psi\mapsto L_\Psi$ is unchanged by bounded perturbations and satisfies
$$
 L_{\Theta\circ\Psi}=L_\Theta L_\Psi
$$
whenever $\Theta:M\to P$ is another quasi-isometry of bounded Abelian defect, where $P$ is a finitely generated torsion-free nilpotent group with a fixed word metric. Every coarse inverse $\overline\Psi:M\to N$, after normalization, has bounded Abelian defect and
$$
 L_{\overline\Psi}=L_\Psi^{-1}.
$$
\end{proposition}

\begin{proof}
Put $a(n):=\pi_M(\Psi(n))$. For every right-invariant mean $\mathsf A$ and every $n\in N$,
$$
 L_{\mathsf A}\pi_N(n)-a(n)
 =\mathsf A_x\bigl[a(xn)-a(x)-a(n)\bigr].
$$
where $\mathsf A_x$ denotes that the mean is being operated considering the expression inside the brackets as a function of $x$. The function inside the mean takes values in the closed $\ell_{S_M}$-ball of radius $D$. Hence \eqref{eq:mean-in-hull} gives
\begin{equation}\label{eq:mean-bounded-defect-estimate}
 \ell_{S_M}\bigl(L_{\mathsf A}\pi_N(n)-a(n)\bigr)\le D
 \qquad(n\in N).
\end{equation}
If two linear maps $L,K:V_N\to V_M$ are at bounded distance from $a$ on $\pi_N(N)$, then, for every $n\in N$ and $k\ge1$,
$$
 k\,\ell_{S_M}\bigl((L-K)\pi_N(n)\bigr)
 =\ell_{S_M}\bigl((L-K)\pi_N(n^k)\bigr)
$$
is bounded independently of $k$. This forces $L\pi_N(n)=K\pi_N(n)$. Since $\pi_N(N)$ spans $V_N$, one has $L=K$. It follows from \eqref{eq:mean-bounded-defect-estimate} that all invariant means give the same map $L_\Psi$. This proves both uniqueness and $\mathcal A(\Psi)=\{L_\Psi\}$.

Set $b(n):=a(n)-L_\Psi\pi_N(n)$. Then $\ell_{S_M}(b(n))\le D$ and
$$
 \Phi_n(x)-L_\Psi\pi_N(n)=b(xn)-b(x),
$$
which proves \eqref{eq:bounded-defect-uniform-increments}. The assertions about the pointed blow-downs now follow from Corollary \ref{cor:uniform-singleton}. Since the first-layer maps of pointed blow-downs of a quasi-isometry are bi-Lipschitz isomorphisms, $L_\Psi$ is invertible.

It remains to record functoriality. If $\Psi'$ is at bounded distance from $\Psi$, then $\pi_M\circ\Psi'-\pi_M\circ\Psi$ is bounded on $N$. Thus, $\pi_M\circ\Psi'-L_\Psi\circ\pi_N$ is bounded on $N$, which implies $\Psi'$ has bounded Abelian defect and by uniqueness gives $L_{\Psi'}=L_{\Psi}$. For composition, write
$$
 \pi_M(\Psi(n))=L_\Psi\pi_N(n)+b_\Psi(n),\qquad
 \pi_P(\Theta(m))=L_\Theta\pi_M(m)+b_\Theta(m),
$$
where $b_\Psi$ is bounded on $N$ and $b_\Theta$ is bounded on $M$. Then
$$
 \pi_P(\Theta(\Psi(n)))
 =L_\Theta L_\Psi\pi_N(n)
  +L_\Theta b_\Psi(n)+b_\Theta(\Psi(n)) \quad(n\in N).
$$
Hence $\pi_P\circ(\Theta\circ\Psi)-L_\Theta L_\Psi\circ\pi_N$ is bounded on $N$, so uniqueness gives $L_{\Theta\circ\Psi}=L_\Theta L_\Psi$.

Finally, let $\overline\Psi$ be a coarse inverse of $\Psi$. Since $\Psi\circ\overline\Psi$ is at bounded distance from the identity, $L_\Psi\pi_N\circ\overline\Psi-\pi_M$ is bounded on $M$. Applying $L_\Psi^{-1}$ shows that $\pi_N\circ\overline\Psi-L_\Psi^{-1}\circ\pi_M$ is bounded on $M$. This implies bounded Abelian defect for $\overline\Psi$, and uniqueness gives $L_{\overline\Psi}=L_\Psi^{-1}$.
\end{proof}

\begin{remark}[Moving basepoints are essential]
\label{rem:moving-basepoints-essential}
Let $\mathcal P=\{0\}\cup\{\pm j^2:j\ge1\}$ and, on $\ZZ$, set
$$
 g(x)=\left\lfloor\frac{\operatorname{dist}(x,\mathcal P)}2\right\rfloor,
 \qquad \Psi(x)=x+g(x).
$$
Then $\Psi$ is a quasi-isometry and $0\le g(x)\le\sqrt{|x|}+1$, so every fixed-basepoint blow-down is the identity. On the two halves of the gap from $j^2$ to $(j+1)^2$, the averages of $\Delta g(x)=g(x+1)-g(x)$ tend respectively to $1/2$ and $-1/2$. Weak-* limits of the corresponding intervals therefore give invariant means with these values. Thus $[1/2,3/2]\subseteq\mathcal A(\Psi)$. Conversely, $|g(x+k)-g(x)|\le |k|/2+1$, and Proposition \ref{prop:localization} gives the reverse inclusion.  Hence
$$
\mathcal A(\Psi)=[1/2,3/2],
$$
although the identity is the unique fixed-basepoint blow-down. In particular, Proposition \ref{prop:bounded-defect-rigidity} shows that this map does not have bounded Abelian defect.
\end{remark}

\begin{proof}[Proof of \Cref{thm:main-pointed}]
Compactness and \eqref{eq:intro-pointed-convexification} are Lemma \ref{lem:retangent-pointed} and Theorem \ref{thm:pointed-convexification}. The singleton characterization of $\mathcal A(\Psi)$ is Corollary \ref{cor:uniform-singleton}.
\end{proof}

\section{Extreme linearizations and cohomological induction}\label{sec:extreme-cohomology}

In this section, for an extreme point $L\in\mathcal A(\Psi)$, we select an ergodic invariant measure in the fiber of a suitable affine map over $L$ and use the associated coupling to extend $L^*$ to a cohomology-ring isomorphism between $H^*(M;\RR)$ and $H^*(N;\RR)$.

Let $\mathscr X_\Psi$ be the closure of $\{\Psi^c:c\in N\}$ in $M^N$, with the product topology. For each fixed $n\in N$, all values $\Psi^c(n)$ lie in one word-metric ball of $M$. Hence, $\mathscr X_\Psi$ is a compact metrizable space. Every $f\in\mathscr X_\Psi$ fixes the identity and satisfies the same quasi-isometry inequalities and coarse-surjectivity bound as $\Psi$. For the last assertion, if $f_j\to f$ pointwise and $m\in M$, choose $n_j$ with $d_M(f_j(n_j),m)$ bounded by the common coarse-surjectivity constant. The lower quasi-isometry inequality confines the $n_j$ to a finite set, so a constant subsequence gives the required preimage for $f$. There is a continuous right action of $N$ on $\mathscr X_\Psi$ given by
\begin{equation}\label{eq:translate-hull-action}
 (f\cdot n)(x):=f(n)^{-1}f(nx) \quad (n,x\in N, f\in\mathscr X_\Psi).
\end{equation}

Let $\mathcal P_N(\mathscr X_\Psi)$ be the weak-$\ast$ compact convex set of $N$-invariant Borel probability measures on $\mathscr X_\Psi$. For $\mu\in\mathcal P_N(\mathscr X_\Psi)$, define
\begin{equation}\label{eq:hull-average-map}
 L_\mu\pi_N(n):=
 \int_{\mathscr X_\Psi}\pi_M(f(n))\,d\mu(f).
\end{equation}
Observe that $L_\mu$ is a homomorphism on $\pi_N(N)$, and hence extends to a linear map $L_\mu:V_N\to V_M$.

\begin{lemma}[Invariant-measure linearizations and invariant-mean linearizations]
\label{lem:hull-image-is-A}
The continuous affine map $\mu\mapsto L_\mu$ has image exactly $\mathcal A(\Psi)$.
\end{lemma}

\begin{proof}
The map $c\mapsto\Psi^c$ is $N$-equivariant. If $\mathsf A$ is a right-invariant mean on $N$, define a state on $C(\mathscr X_\Psi)$ by
$$
 u\longmapsto \mathsf A_c[u(\Psi^c)].
$$
where $\mathsf A_c$ denotes that the mean is being operated on the function $c\mapsto u(\Psi^c)$. The Riesz representation theorem gives an invariant probability measure $\mu$ on $\mathscr X_\Psi$, and
$$
 L_\mu\pi_N(n)
 =\mathsf A_c[\pi_M(\Psi^c(n))]
 =\mathsf A_c[\Phi_n(c)]
 =L_{\mathsf A}\pi_N(n).
$$
Conversely, let $\mu$ be invariant. The map $c\mapsto\Psi^c$ induces a unital positive isometry
$$
 T:C(\mathscr X_\Psi)\longrightarrow\ell^\infty(N),
 \qquad (Tu)(c):=u(\Psi^c),
$$
because $\{\Psi^c:c\in N\}$ is dense in $\mathscr X_\Psi$. The state $Tu\mapsto\int_{\mathscr X_\Psi} u\,d\mu$ on the range of $T$ extends, by Hahn--Banach theorem, to a state $\eta$ on $\ell^\infty(N)$. Fix a right-invariant mean $\mathsf m$ on $N$ and set
$$
 \mathsf A(f):=
 \mathsf m_g\left[\eta(f(\,\cdot\, g))\right] \quad(f\in \ell^\infty(N)),
$$
where again $\mathsf m_g$ denotes that the mean is being operated on the function $g\mapsto \eta(f\,\cdot\,g)$. Then, $\mathsf A$ is right invariant. If $f=Tu$, invariance of $\mu$ and \eqref{eq:translate-hull-action} give
$$
 \eta(f(\,\cdot\,g))
 =\eta(Tu_g)
 =\int_{\mathscr X_\Psi}u(F\cdot g)\,d\mu(F)
 =\int_{\mathscr X_\Psi}u(F)\,d\mu(F).
$$
where $u_g(F):=u(F\cdot g)$, $F\in\mathscr X_\Psi, g\in N$. Thus, $\mathsf A(f)=\int_{\mathscr X_\Psi}u(F)\,d\mu(F)$ when $f=Tu$. Applying this by taking $u$ to be the coordinate functions of the map $F\mapsto\pi_M(F(n))$, $n\in N$ gives $L_{\mathsf A}=L_\mu$.
\end{proof}

For the remaining portion of this section, we require the notion of uniform measure equivalence. We shall be following the definition of uniform measure equivalence given in \cite{GotfredsenKyed21}.

\begin{proposition}[The coupling associated with an ergodic measure]
\label{prop:hull-ume-coupling}
If $\mu\in\mathcal P_N(\mathscr X_\Psi)$ is ergodic, then
$$
 (\Omega,\eta):=
 (\mathscr X_\Psi\times M,\,\mu\times\text{counting measure})
$$
is an ergodic uniform measure-equivalence coupling of $M$ and $N$ for the commuting left actions
\begin{align*}
 m_0\cdot(f,m)&=(f,m_0m),\\
 n\star(f,m)&=(f\cdot n^{-1},mf(n^{-1})).
\end{align*}
\end{proposition}

\begin{proof}
Clearly the actions of $M$ and $N$ on $\Omega$ commute and preserve $\eta$, which makes $(\Omega,\eta)$ a Borel measure $M\times N$-space. Take $Y=\mathscr X_\Psi\times\{e_M\}$ and $\nu=\eta|_Y$. Let $R$ be a common coarse-surjectivity constant for the maps in $\mathscr X_\Psi$. Every $N$-orbit in $\Omega$ meets $\mathscr X_\Psi\times B_M(e,R)$: given $(f,m)$, choose $n$ with $d_M(f(n),m^{-1})\le R$. The intersection of any orbit with $\mathscr X_\Psi\times B_M(e,R)$ is finite. Indeed, if both $(f,a)$ and $n^{-1}\star(f,a)=(f\cdot n,af(n))$ have second coordinate in $B_M(e,R)$, then $|f(n)|_{S_M}\le2R$, and the common lower quasi-isometry inequality bounds $|n|_{S_N}$ uniformly. The resulting finite Borel equivalence relation on $\mathscr X_\Psi\times B_M(e,R)$ admits a Borel transversal $X$. Set $\mu_X=\eta|_X$. 

Define $i:M\times Y\to \Omega$ by $i(m,(f,e_M))=(f,m)$ and $j:N\times X\to \Omega$ by $j(n,y)=n\star y$ for $y\in X$. The $N$-action is free: if $n\star(f,m)=(f,m)$, then $f(n^{-k})=e_M$ for every $k\in\mathbb Z$. The lower quasi-isometry inequality bounds $|n^k|_{S_N}$ uniformly, so $n$ has finite order and hence $n=e_N$ by torsion-freeness. Since the $N$-action is free and $X$ meets each orbit once, $j$ is a Borel bijection with Borel inverse. The disjoint decomposition $\Omega=\bigsqcup_{n\in N}n\star X$ and invariance of $\eta$ show that $j$ transports counting measure times $\mu_X$ to $\eta$. Moreover, $0<\mu_X(X)\le |B_M(e,R)|<\infty$; positivity follows from $\eta\ne0$ and the countable decomposition of $\Omega$ above. The map $i$ similarly transports counting measure times $\nu$ to $\eta$, so these maps give strict measure-equivalence isomorphisms. Let $(\Lambda,A)$ be the common quasi-isometry constants for the maps in $\mathscr X_\Psi$. If $j^{-1}\circ i(m,f)=(n,y)$, then
$$
 y=n^{-1}\star(f,m)=(f\cdot n,mf(n))\in X.
$$
Thus $|f(n)|_{S_M}\le |m|_{S_M}+R$, and the lower quasi-isometry inequality gives $|n|_{S_N}\le\Lambda(|m|_{S_M}+R+A)$. Conversely, if $y=(f,a)\in X$, then
$$
 i^{-1}\circ j(n,y)=(af(n^{-1}),(f\cdot n^{-1},e_M)),
 \qquad
 |af(n^{-1})|_{S_M}\le R+\Lambda|n|_{S_N}+A.
$$
Consequently, for every $t>0$,
\begin{align*}
 j^{-1}\circ i(B_M(e,t)\times Y)
 &\subseteq B_N(e,\Lambda(t+R+A))\times X,\\
 i^{-1}\circ j(B_N(e,t)\times X)
 &\subseteq B_M(e,R+\Lambda t+A)\times Y.
\end{align*}
Thus, for every compact $C \subset M$, there is some $t > 0$ such that $C \subseteq B_M(e, t)$, and $D = B_N(e, \Lambda(t+R+A))$ is a compact subset of $N$ satisfying $j^{-1}\circ i(C \times Y) \subseteq D \times X$. Similarly, for every compact $D \subset N$, there is some $t > 0$ such that $D \subseteq B_N(e,t)$, and $C = B_M(e, R+\Lambda t+A)$ is a compact subset of $M$ satisfying $i^{-1}\circ j(D \times X) \subseteq C \times Y$. This shows that $(\Omega,\eta)$ is a uniform measure-equivalence coupling of $M$ and $N$.

A measurable set $W\subseteq \Omega$ invariant under both groups is, by $M$-invariance, of the form $W=A\times M$ for some measurable set $A\subseteq\mathscr X_\Psi$. Since $W$ is $N$-invariant, we have $A$ is also $N$-invariant. Ergodicity of $\mu$ then proves ergodicity of the coupling.
\end{proof}

We identify
$$
 H^1(N;\RR)=\Hom(N,\RR)=V_N^*,
 \qquad
 H^1(M;\RR)=V_M^*.
$$

\begin{theorem}[Cohomological compatibility of extreme points]
\label{thm:extreme-cohomology}
For every $L\in\operatorname{Ext}\mathcal A(\Psi)$, there is an isomorphism of graded unital real algebras
\begin{equation}\label{eq:extreme-ring-isomorphism}
 \mathcal I_L:H^*(M;\RR)\longrightarrow H^*(N;\RR)
\end{equation}
whose restriction to degree one is
\begin{equation}\label{eq:degree-one-is-L-star}
 \mathcal I_L|_{H^1}=L^*:V_M^*\longrightarrow V_N^*.
\end{equation}
\end{theorem}

\begin{proof}
By \Cref{lem:hull-image-is-A}, the fiber
$$
 \mathcal F_L:=
 \{\mu\in\mathcal P_N(\mathscr X_\Psi):L_\mu=L\}
$$
is nonempty, compact, and convex. Let $\mu$ be an extreme point of $\mathcal F_L$, which exists by Krein-Milman theorem. Since $L$ is an extreme point in $\mathcal A(\Psi)$, $\mu$ is an extreme point of $\mathcal P_N$ and hence is ergodic.

By \Cref{prop:hull-ume-coupling}, $\mu$ gives an ergodic uniform measure-equivalence coupling of $M$ and $N$. Finitely generated torsion-free nilpotent groups have property $H_T^k$ for every $k$ and their real cohomology groups are finite-dimensional, hence ordinary and reduced cohomology agree \cite{GotfredsenKyed21}.  The induction construction of Gotfredsen--Kyed \cite{GotfredsenKyed21}, applied to this coupling, gives the graded ring isomorphism \eqref{eq:extreme-ring-isomorphism}. We take $\mathcal I_L$ to be the particular isomorphism given by that construction and identify $\mathcal I_L|_{H^1}$.

For trivial real coefficients, an inhomogeneous $1$-cocycle on a (discrete) group $G$ is a function $q:G\to\RR$ satisfying $q(gh)=q(g)+q(h)$. Since every $1$-coboundary is zero, we have $H^1(G;\RR)=\Hom(G,\RR)$. The corresponding invariant homogeneous cocycle is $(g_0,g_1)\mapsto q(g_0^{-1}g_1)$. Conversely, the real-valued homomorphism represented by a invariant homogeneous cocycle $\zeta$ is $g\mapsto\zeta(e_G,g)$. In the induction construction of Gotfredsen--Kyed \cite{GotfredsenKyed21}, $\mathcal I_L|_{H^1}$ is induced by $\kappa^1$. Use $G=M$ and $H=N$ in the notation of Gotfredsen--Kyed \cite{GotfredsenKyed21}, and identify $Y$ with $\mathscr X_\Psi$. Their coupling cocycle $\omega_M:N\times\mathscr X_\Psi\to M$ is defined by
$$
 n\star i(m,f)=i\bigl(m\omega_M(n,f)^{-1},f\cdot n^{-1}\bigr).
$$

For $\chi\in V_M^*$, define
$$
 \xi_\chi(m_0,m_1)
 =\chi\bigl(\pi_M(m_0^{-1}m_1)\bigr).
$$
For $n_0,n_1\in N$, we have
\begin{align*}
 \kappa^1(\xi_\chi)(n_0,n_1)
 &=\int_{\mathscr X_\Psi}
   \chi\bigl(\pi_M(f(n_0)^{-1}f(n_1))\bigr)\,d\mu(f)\\
 &=\int_{\mathscr X_\Psi}
   \chi\bigl(\pi_M((f\cdot n_0)(n_0^{-1}n_1))\bigr)\,d\mu(f)\\
 &=\int_{\mathscr X_\Psi}
   \chi\bigl(\pi_M(f(n_0^{-1}n_1))\bigr)\,d\mu(f)\\
 &=\chi\bigl(L_\mu\pi_N(n_0^{-1}n_1)\bigr)\\
 &=\chi\bigl(L\pi_N(n_0^{-1}n_1)\bigr).
\end{align*}
The element in $H^1(N;\RR)=\Hom(N,\RR)$ representing $\mathcal I_L(\chi)$ is therefore
$$
 n\longmapsto\kappa^1(\xi_\chi)(e_N,n)
 =\chi\bigl(L\pi_N(n)\bigr).
$$
By definition, $L^*\chi$ is the functional $v\mapsto\chi(Lv)$ on $V_N$, whose corresponding element in $\Hom(N,\RR)$ is $n\mapsto\chi(L\pi_N(n))$. This proves $\mathcal I_L|_{H^1}=L^*$ and \eqref{eq:degree-one-is-L-star}.
\end{proof}

\begin{corollary}[Cohomology in the bounded-defect regime]
\label{cor:bounded-defect-cohomology}
If $\Psi$ has bounded Abelian defect, then there is an isomorphism of graded unital real algebras
$$
 \mathcal I:H^*(M;\RR)\longrightarrow H^*(N;\RR)
$$
whose restriction to degree one is
$$
 \mathcal I|_{H^1}=L_\Psi^*:V_M^*\longrightarrow V_N^*.
$$
Its inverse restricts in degree one to $(L_\Psi^{-1})^*:V_N^*\to V_M^*$.
\end{corollary}

\begin{proof}
By \Cref{prop:bounded-defect-rigidity}, the singleton set $\mathcal A(\Psi)=\{L_\Psi\}$ has $L_\Psi$ as an extreme point. The claim is therefore \Cref{thm:extreme-cohomology}, together with the degree-one restriction of its inverse.
\end{proof}

\begin{remark}
The maps $L_\Psi^*$ and $(L_\Psi^{-1})^*$ are canonical and independent of the invariant mean. 
\end{remark}

\begin{proof}[Proof of \Cref{thm:main-extreme-cohomology}]
This is \Cref{thm:extreme-cohomology}.
\end{proof}

\section{Isometric asymptotic cones with a multiplicative constant gap}
\label{sec:metric-gap}

We apply \Cref{thm:extreme-cohomology} to two word metrics on one nilpotent group. Their asymptotic cones are isometric, but the multiplicative constants of all quasi-isometries between them are bounded away from one.

Let $\mathfrak h$ be the seven-dimensional real nilpotent Lie algebra whose Chevalley--Eilenberg differential, in a basis $a,b,c,d,e,f,g$ of $\mathfrak h^*$, is
\begin{equation}\label{eq:seven-dimensional-algebra}
 da=db=dc=df=0,
 \qquad dd=ab,
 \qquad de=ad+bc,
 \qquad dg=ae-cd+cf.
\end{equation}
Let $H$ be the corresponding simply connected nilpotent Lie group. The center of $\mathfrak h$ is $\RR X_7$, so $\mathfrak h$ is indecomposable, since a direct sum of two nonzero nilpotent Lie algebras has center of dimension at least two.

Let $X_1,\ldots,X_7$ be the basis dual to $a,b,c,d,e,f,g$. The lower-central-series weights are
$$
 1,1,1,2,3,1,4,
$$
and the associated graded Lie algebra is
\begin{equation}\label{eq:seven-dimensional-graded}
 \operatorname{gr}(\mathfrak h)
 =(0,0,0,12,14,0,15).
\end{equation}
Its first layer is
$$
 V=\operatorname{span}\{X_1,X_2,X_3,X_6\}.
$$
Define a graded automorphism $P$ of $\operatorname{gr}(\mathfrak h)$ by
\begin{equation}\label{eq:graded-shear-P}
 P(X_6)=X_6+X_3,
 \qquad P(X_i)=X_i\quad(i\ne6).
\end{equation}
This is well defined because $X_3$ and $X_6$ are central in the graded Lie algebra. On first cohomology,
\begin{equation}\label{eq:P-star-on-c}
 P^*c=c+f.
\end{equation}

\begin{lemma}[The cohomology obstruction]
\label{lem:seven-dimensional-cohomology}
For $\alpha\in H^1(\mathfrak h)$, let
$$
 m_\alpha:H^2(\mathfrak h)\longrightarrow H^3(\mathfrak h),
 \qquad u\longmapsto\alpha\wedge u.
$$
Then
\begin{equation}\label{eq:cohomology-rank-jump}
 \operatorname{rank}m_c=1,
 \qquad
 \operatorname{rank}m_{c+f}=2.
\end{equation}
More generally, if $\alpha=Aa+Bb+Cc+Ff$ and $F\ne0$, then $\operatorname{rank}m_\alpha\ge2$. Consequently, if a linear isomorphism $L:V\to V$ has the property that $L^*:H^1(\mathfrak h)\to H^1(\mathfrak h)$ extends to an automorphism of $H^*(\mathfrak h)$, then
\begin{equation}\label{eq:ring-compatible-hyperplane}
 c(LX_6)=0.
\end{equation}
\end{lemma}

\begin{proof}
Using \eqref{eq:seven-dimensional-algebra}, it is a routine calculation to find out the following bases:
\begin{align*}
 H^1(\mathfrak h):\quad&\operatorname{span}\{[a],[b],[c],[f]\}\\
 H^2(\mathfrak h):\quad& [ac],[ad],[af],[bd],[bf],[-ae+cd],\\
 H^3(\mathfrak h):\quad& [abe],[acg],[-adg+afg],[-ade+bce],\\
 & [bde],[bdf],[bcg+cde].
\end{align*}
We also observe that
$$
 \operatorname{rank}(d:\Lambda^1\to\Lambda^2)=3,
 \quad
 \operatorname{rank}(d:\Lambda^2\to\Lambda^3)=12,
 \quad
 \operatorname{rank}(d:\Lambda^3\to\Lambda^4)=16.
$$
The relations
$$
 d(de)=abe-bcd,
 \quad d(bg)=abe+bcd-bcf,
 \quad d(ef)=adf+bcf
$$
give
\begin{equation}\label{eq:key-cup-products}
 [f\wedge ad]=-2[abe],
 \quad[c\wedge bd]=-[abe],
 \quad[f\wedge bd]=[bdf].
\end{equation}
The table below illustrates the action of $m_\alpha$ on the basis elements of $H^2(\mathfrak h)$:
$$
\begin{array}{c|cccccc}
u&[ac]&[ad]&[af]&[bd]&[bf]&[-ae+cd]\\ \hline
m_\alpha(u)&0&-2F[abe]&0&-C[abe]+F[bdf]&-2C[abe]&2B[abe].
\end{array}
$$
Thus, $m_c$ has image exactly $\RR[abe]$, so $\operatorname{rank}m_c=1$. If $F\ne 0$, the images of $[ad]$ and $[bd]$ displayed in the table are linearly independent. This proves that $\operatorname{rank}m_\alpha\ge 2$. Further, if only $C$ and $F$ are non-zero then it is evident from the above table that $\operatorname{rank} m_{c+f}=2$.

Any graded ring automorphism of $H^*(\mathfrak h)$ preserves the rank of $m_\alpha$. If $L^*$ extends to such an automorphism, then $\operatorname{rank}m_{L^*c}=\operatorname{rank}m_c=1$. The $f$-coefficient of $L^*c$ must therefore vanish, which implies $c(LX_6)=(L^*c)(X_6)=0$, proving \eqref{eq:ring-compatible-hyperplane}.
\end{proof}

Let
$$
 u_1=X_1,\quad u_2=X_2,\quad u_3=X_3,\quad u_4=X_6,
$$
and
\begin{equation}\label{eq:asymmetric-polytope}
 v:=2u_1+3u_2+5u_3+7u_4,
 \qquad
 B:=\operatorname{conv}\{\pm u_1,\ldots,\pm u_4,\pm v\}.
\end{equation}
Define $\operatorname{Stab}(B)=\{ T\in\operatorname{GL}(V):T(B)=B\}$.

\begin{lemma}[Extreme points and stablilizer]
\label{lem:asymmetric-stable-ball}
Every element of $\{\pm u_1,\ldots,\pm u_4,\pm v\}$ is an extreme point of $B$, and $\operatorname{Stab}(B)=\{I,-I\}$.
\end{lemma}

\begin{proof}
Set $u_5:=v$. Showing that $u_i$ is an extreme point of $B$ is equivalent to showing that $u_i\notin \operatorname{conv}\left(\{\pm u_1,\dots\pm u_4,\pm v\}\setminus \{u_i\}\right)$. Suppose that some $u_i$ is not an extreme point of $B$. Then, 
$$ 
u_i = \alpha(-u_i) + \sum_{j \neq i} c_j u_j 
$$
where $\alpha \ge 0$, $c_j \in \mathbb{R}$, and $\alpha + \sum |c_j| \le 1$. Therefore, 
$$
(1 + \alpha) u_i = \sum_{j \neq i} c_j u_j \implies u_i = \sum_{j \neq i} d_j u_j \quad\left(d_j=\frac{c_j}{1 + \alpha}\right). 
$$
and
$$ 
\sum_{j\neq i}|d_j|=\sum_{j \neq i} \left| \frac{c_j}{1 + \alpha} \right| = \frac{\sum |c_j|}{1 + \alpha} \le \frac{1 - \alpha}{1 + \alpha}\le 1 
$$
Via the relation in \eqref{eq:asymmetric-polytope}, it is easy to see that the above cannot hold for any $u_i$. So, every element of $\{\pm u_1,\dots,\pm u_4, \pm v \}$ is an extreme point of $B$.

Let $T \in \operatorname{GL}(V)$ be such that $T(B) = B$. From \eqref{eq:asymmetric-polytope}, we have 
$$
 \sum_{k=1}^5 c_k u_k = 0, \quad (c_1, c_2, c_3, c_4, c_5) = (-2, -3, -5, -7, 1).
$$
Notice that the absolute values of these coefficients $\{1, 2, 3, 5, 7\}$ are pairwise distinct. Since $T$ is an invertible linear map, it permutes the extreme points of $B$. Hence, 
$$
 T(u_k) = \epsilon_k u_{\sigma(k)}
$$
where $\sigma$ is a permutation of the set $\{1,2,3,4,5\}$ and $\epsilon_k \in \{1, -1\}$. Applying the linear map $T$ to the relation $\sum_{k=1}^5 c_k u_k = 0$ yields
$$
 \sum_{k=1}^5 c_k \epsilon_k u_{\sigma(k)} = 0
$$
By re-indexing the sum via $j = \sigma(k)$ (so $k = \sigma^{-1}(j)$), we get 
$$
 \sum_{j=1}^5 \big(c_{\sigma^{-1}(j)} \epsilon_{\sigma^{-1}(j)}\big) u_j = 0
$$
Since $u_1,u_2,u_3,u_4$ are linearly independent, there exists some real number $\lambda \neq 0$ such that for all $j$
$$
 c_{\sigma^{-1}(j)} \epsilon_{\sigma^{-1}(j)} = \lambda c_j \implies |c_{\sigma^{-1}(j)}| = |\lambda| |c_j|.
$$
This condition forces $|\lambda| = 1$. and $|c_{\sigma^{-1}(j)}| = |c_j|$ for all $j$. Since the numbers in $\{1, 2, 3, 5, 7\}$ are pairwise distinct, $\sigma$ has to be the identity permutation. Hence, for each $j$ we get
$$
 c_j \epsilon_j = \lambda c_j
$$
As none of the coefficients $c_j$ are zero, we can divide by $c_j$ to obtain:
$$
 \epsilon_j = \lambda \qquad \text{for all } j=1,2,3,4,5.
$$
If $\lambda = 1$, then $\epsilon_j = 1$ for all $j$, which means $T = I$. If $\lambda = -1$, then $\epsilon_j = -1$ for all $j$, which means $T = -I$. Therefore, $\operatorname{Stab}(B)=\{I, -I\}$.
\end{proof}

Put $\mathfrak h_{\mathbb Q}:=\operatorname{span}_{\mathbb Q}\{X_1,\ldots,X_7\}$ and $\mathcal L:=\operatorname{span}_{\mathbb Z}\{X_1,\ldots,X_7\}$. The structure constants in \eqref{eq:seven-dimensional-algebra} are integers, so $\mathfrak h_{\mathbb Q}$ is a Lie algebra over $\mathbb Q$. By Malcev's theorem \cite[Theorem 2.12 and the paragraph that follows]{Raghunathan72}, the subgroup $\Gamma:=\langle\exp_H(\mathcal L)\rangle$ is a lattice in $H$ and
$$
\operatorname{span}_{\mathbb Q}(\log\Gamma)=\mathfrak h_{\mathbb Q}.
$$
Thus every $\log\gamma$, $\gamma\in\Gamma$, has rational coordinates in the basis $X_1,\ldots,X_7$, and conversely every $X_i$ is a finite rational linear combination of elements of $\log\Gamma$.

Let $\pi:\Gamma\to\Gamma/[\Gamma,\Gamma]$ be the canonical projection and $p:\mathfrak h\to V$ be the linear projection with kernel $[\mathfrak h,\mathfrak h]=\operatorname{span}_{\mathbb R}\{X_4,X_5,X_7\}$. Explicitly,
$$
p\left(\sum_{i=1}^7 t_iX_i\right)
=t_1u_1+t_2u_2+t_3u_3+t_6u_4.
$$
Under the identification $\Gamma/[\Gamma,\Gamma]\otimes_\ZZ\RR\cong V$ via the map sending $\pi(\gamma)\otimes r$ to $rp(\log \gamma)$, we have $\pi(\gamma)\otimes 1=p(\log\gamma)$ for every $\gamma\in\Gamma$. Now,
$$
\operatorname{span}_{\mathbb Q}\{\pi(\gamma)\otimes1:\gamma\in\Gamma\}
=p\bigl(\operatorname{span}_{\mathbb Q}(\log\Gamma)\bigr)
=p(\mathfrak h_{\mathbb Q})
=\operatorname{span}_{\mathbb Q}\{u_1,\ldots,u_4\}.
$$

Thus, if $w_1,\ldots,w_4$ is a $\mathbb Z$-basis of the free part of $\pi(\Gamma)$, each $u_i$ has rational coordinates in this basis:
$$
u_i=\sum_{j=1}^4 a_{ij}(w_j\otimes 1),
\qquad a_{ij}\in\mathbb Q.
$$
Choose a positive integer $q_0$ divisible by the denominators of all the $a_{ij}$. Then $q_0u_i$ is an integer linear combination of the $w_j\otimes 1$, so $q_0u_i\in\{\pi(\gamma)\otimes 1:\gamma\in\Gamma\}$ for every $i$.

The set of extreme points of $B$ and $P(B)$ are $\{\pm u_1,\dots,\pm u_4,\pm v\}$ and $\{\pm u_1,\pm u_2,\pm u_3,\pm (u_4+u_3),\pm P(v) \}$. Hence, every extreme point of $B$ and $P(B)$ is an integer linear combination of the $u_i$. Consequently, whenever $q$ is a positive integer multiple of $q_0$, every extreme point of $qB$ and $qP(B)$ belongs to $\{\pi(\gamma)\otimes 1:\gamma\in\Gamma\}$.

Choose a finite symmetric generating set $R$ of $\Gamma$. Choose a positive integer multiple $q$ of $q_0$ large enough that the set $\{\pi(\gamma)\otimes 1:\gamma\in R\}$ lies in both $qB$ and $qP(B)$. For each $w\in\{u_1,\ldots,u_4,v\}$, choose $s_w,t_w\in\Gamma$ with $\pi(s_w)\otimes1=qw$ and $\pi(t_w)\otimes1=qP(w)$. Let $S=R\cup\{s_w^{\pm1}:w\in\{u_1,\ldots,u_4,v\}\}$ and $T=R\cup\{t_w^{\pm1}:w\in\{u_1,\ldots,u_4,v\}\}$. Then $S,T$ are finite symmetric generating sets of $\Gamma$, and the unit balls of $\ell_S$ and $\ell_T$ are $qB$ and $qP(B)$ respectively. The graded automorphism $P$ is an isometry between the asymptotic cones of $(\Gamma, S)$ and $(\Gamma, T)$.

For symmetric, compact, convex sets $K,K'$ in $V$ containing the origin in their interiors, write
$$
\|D\|_{K\to K'}:=\inf\{t>0:D(K)\subseteq tK'\}.
$$
Equivalently,
$$
\|D\|_{K\to K'}=\sup_{v\in V\setminus\{0\}}\frac{\|Dv\|_{K'}}{\|v\|_K},
$$
where $\|\cdot\|_K$ and $\|\cdot\|_{K'}$ are the Minkowski functionals corresponding to $K$ and $K'$ respectively. It is easy to see that $\|D\|_{K\to K'}\le1$ implies $D(K)\subseteq K'$. Further, define
\begin{equation}\label{eq:def-lambda-zero}
 \lambda_0:=
 \inf_{\substack{D\in\operatorname{GL}(V)\\c(DX_6)=0}}
 \max\left\{
  \|D\|_{B\to P(B)},
  \|D^{-1}\|_{P(B)\to B}
 \right\}.
\end{equation}

\begin{remark}\label{rem:first-layer-identification}
The identification $\Gamma/[\Gamma,\Gamma]\otimes_\ZZ\RR\cong V$ is an isometry between the normed spaces $\left(\Gamma/[\Gamma,\Gamma]\otimes_\ZZ\RR,\ell_S\right)$ and $(V,\|\cdot\|_{qB})$, and also between $\left(\Gamma/[\Gamma,\Gamma]\otimes_\ZZ\RR,\ell_T\right)$ and $(V,\|\cdot\|_{qP(B)})$.    
\end{remark}

\begin{lemma}\label{lem:lambda-zero-gap}
One has $\lambda_0>1$.
\end{lemma}

\begin{proof}
Observe that for every $D\in\operatorname{GL}(V)$,
$$
 \|D\|_{B\to P(B)}
 \|D^{-1}\|_{P(B)\to B}\ge1,
$$
so the maximum in \eqref{eq:def-lambda-zero} is at least one and $\lambda_0\ge1$. If $\lambda_0=1$, there is a sequence of maps $(D_j)$ in $\operatorname{GL}(V)$ such that $c(D_jX_6)=0$ and $\max\left\{\|D_j\|_{B\to P(B)},\|D_j^{-1}\|_{P(B)\to B}\right\}\to 1$ as $j\to\infty$. Since a convergent sequence is bounded, $\left(\|D_j\|_{B\to P(B)}\right)$ and $\left(\|D_j^{-1}\|_{P(B)\to B}\right)$ are bounded sequences, so after passing to a subsequence we have $D_j\to D$ and $D_j^{-1}\to D'$ in $\operatorname{End}(V)$. Taking limits in $D_jD_j^{-1}=D_j^{-1}D_j=I$ gives $DD'=D'D=I$. Thus $D\in\operatorname{GL}(V)$ and $D'=D^{-1}$. Since
\begin{align*}
\|D\|_{B\to P(B)}&=\lim_{j\to\infty}\|D_j\|_{B\to P(B)}\le 1, \text{ and }\\ 
\|D^{-1}\|_{P(B)\to B}&=\lim_{j\to\infty}\|D_j^{-1}\|_{P(B)\to B}\le 1,
\end{align*} 
we have $D(B)=P(B)$. Then $P^{-1}D(B)=(B)$, and Lemma \ref{lem:asymmetric-stable-ball} gives $D=\pm P$. But
$$
 c(PX_6)=1,
 \qquad c(-PX_6)=-1,
$$
contradicting that $c(DX_6)=\lim_{j\to\infty}c(D_jX_6)=0$.
\end{proof}

For word metrics $d_S,d_T$ on the same group, put
$$
 \begin{aligned}
 \lambda_{\mathrm{QI}}(d_S,d_T):=\inf\{\Lambda\ge1:
 &\text{ there is a $(\Lambda,A)$-quasi-isometry}\\
 &(\Gamma,d_S)\to(\Gamma,d_T)
 \text{ for some }A\ge0\}.
 \end{aligned}
$$

\begin{theorem}[A multiplicative constant gap]\label{thm:positive-multiplicative-gap}
The two word metrics $d_S,d_T$ constructed above have isometric asymptotic cones, but
\begin{equation}\label{eq:positive-multiplicative-gap}
 \lambda_{\mathrm{QI}}(d_S,d_T)\ge\lambda_0>1.
\end{equation}
\end{theorem}

\begin{proof}
Let $\Psi:(\Gamma,d_S)\to(\Gamma,d_T)$ be a $(\Lambda,A)$-quasi-isometry. We shall keep in mind Remark \ref{rem:first-layer-identification} which will allow us to treat elements of $\mathcal A(\Psi)$ as maps on $V$ with their operator norms unchanged. For every extreme point $E\in\mathcal A(\Psi)$, Theorem \ref{thm:extreme-cohomology} gives a ring isomorphism whose degree-one part is $E^*$. Let $M=\Gamma\backslash H$. Then, $M$ is $K(\Gamma,1)$, and hence $H^*(\Gamma; \mathbb{R}) \cong H^*(M; \mathbb{R})$. Since $M$ is a smooth manifold, we have $H^*(M;\RR)\cong H_{dR}^*(M)$, where $H_{dR}^*(M)$ denotes the de Rham cohomology of $M$. By \cite{Nomizu54}, we have $H^*_{dR}(M) \cong H^*(\mathfrak{h})$. Hence,
$$
H^*(\Gamma; \mathbb{R}) \cong H^*(M; \mathbb{R}) \cong H^*_{dR}(M) \cong H^*(\mathfrak{h}).
$$
These graded algebra isomorphisms identify $\alpha\in V^*$ with $\gamma\mapsto\alpha(p(\log\gamma))$ by integration along $t\mapsto\Gamma\exp_H(t\log\gamma)$, $0\le t\le1$, so their degree-one identification is the one used above. Therefore, Lemma \ref{lem:seven-dimensional-cohomology} gives
$$
 c(EX_6)=0
 \qquad(E\in\operatorname{Ext}\mathcal A(\Psi)).
$$
Since $c$ is linear and $\mathcal A(\Psi)$ is a compact convex set, by Krein-Milman theorem it holds for every $E\in\mathcal A(\Psi)$.

Choose a fixed-basepoint blow-down of $\Psi$ and a Pansu differentiability point. Its first-layer differential $D$ belongs to $\mathcal A(\Psi)$ by Corollary \ref{cor:mean-realization}, therefore $c(DX_6)=0$. By Lemma \ref{lem:first-layer-biLip},
$$
 \|D\|_{B\to P(B)}=\|D\|_{qB\to qP(B)}\le\Lambda,
 \qquad
 \|D^{-1}\|_{P(B)\to B}=\|D^{-1}\|_{qP(B)\to qB}\le\Lambda.
$$
Definition \eqref{eq:def-lambda-zero} now gives $\Lambda\ge\lambda_0$. Taking the infimum over all quasi-isometries proves \eqref{eq:positive-multiplicative-gap}.
\end{proof}

\begin{proof}[Proof of \Cref{thm:main-multiplicative-gap}]
The lattice and its two word metrics are constructed in Section \ref{sec:metric-gap}, and the uniform gap \eqref{eq:uniform-multiplicative-gap} is Theorem \ref{thm:positive-multiplicative-gap}, along with the observation that $\lambda_{\mathrm{QI}}(d_T,d_S)\ge\lambda_0$, since a coarse inverse of a $(\Lambda,A)$-quasi-isometry has the same multiplicative constant after changing the additive constant.
\end{proof}

\appendix
\section{}\label{app:cone-compactness}

\subsection{Density of rescaled lattices and a properness estimate}
\label{app:cone-density-details}

In this subsection, we provide the proofs of \eqref{eq:cone-density} and \eqref{eq:cone-bounded-lift}. Choose a compact set $K_G\subseteq\widetilde G$ with $\widetilde G=GK_G$, and transport the multiplication of $\widetilde G$ to $G_\infty$ by defining
$$
 u*_r^Gv:=
 \sigma_r^G\left((\sigma_r^G)^{-1}(u)(\sigma_r^G)^{-1}(v)\right).
$$
The Baker--Campbell--Hausdorff formula is finite. In adapted lower-central series coordinates, a nested bracket whose input components have total weight $d$ and whose value has weight $k\ge d$ acquires the factor $r^{d-k}$. The components with $k>d$ therefore vanish in the limit, while those with $k=d$ give the bracket of the associated graded Lie algebra. Hence $*_r^G$ converges locally uniformly to multiplication in $G_\infty$.

Given $z\in\overline {B_{G,\infty}(e,R)}$, write $(\sigma_r^G)^{-1}(z)=ak$ with $a\in G$ and $k\in K_G$. Then
$$
 \sigma_r^G(a)=z*_r^G\sigma_r^G(k^{-1}).
$$
Since $\sigma_r^G(k^{-1})\to e$ uniformly for $k\in K_G$, the right-hand side converges to $z$ uniformly on $\overline {B_{G,\infty}(e,R)}$. This proves \eqref{eq:cone-density}.

For \eqref{eq:cone-bounded-lift}, fix norms on the spaces $W_i$ and write $\log_{\widetilde G}a=Y_1+\cdots+Y_s$, with $Y_i\in W_i$. Then,
$$
\log_\infty\bigl(\sigma_r^G(a)\bigr)
=\sum_{i=1}^s r^{-i}q_i(Y_i).
$$
The set $\log_\infty(\overline {B_{G,\infty}(e,R)})$ is compact, so each of its coordinate projections is bounded. Since each $q_i:W_i\to V_i^G$ is a linear isomorphism, there are constants $A_{i,R}>0$, independent of $a$ and $r$, such that
$$
\|Y_i\|\le A_{i,R}r^i
\qquad(1\le i\le s)
$$
whenever $\sigma_r^G(a)\in\overline {B_{G,\infty}(e,R)}$. Now, there exists a constant $C>0$ such that
$$
|a|_{S_G}\le C\left(1+\max_{1\le i\le s}\|Y_i\|^{1/i}\right)
\qquad(a\in G);
$$
see \cite[proof of Th\'eor\`eme II.1]{Guivarch1973}, \cite[Theorem 2.7]{Breuillard2014}. Consequently, for $r\ge1$,
$$
|a|_{S_G}
\le C\left(1+r\max_{1\le i\le s}A_{i,R}^{1/i}\right)
\le C_Rr,
\qquad
C_R:=C\left(1+\max_{1\le i\le s}A_{i,R}^{1/i}\right).
$$
This proves \eqref{eq:cone-bounded-lift}.

\subsection{Proof of the existence of cone maps}
\label{app:cone-map-proof}

\begin{proof}[Proof of \Cref{lem:cone-map}]
For $G\in\{N,M\}$ set
$$
 \theta_j^G(R):=
 \sup_{z\in\overline {B_{G,\infty}(e,R)}}
 d_{G,\infty}(z,D_j^G).
$$
Then $\theta_j^G(R)\to0$ for each fixed $R$ by \eqref{eq:cone-density}.

Fix $R>0$ and let $x=\sigma_{r_j}^N(a)$ and $y=\sigma_{r_j}^N(b)$ belong to $D_j^N\cap\overline B_{N,\infty}(e,R)$. By \eqref{eq:cone-bounded-lift},
$$
 |a|_{S_N},|b|_{S_N}\le C_Rr_j.
$$
The normalization $\Psi(e)=e$ and the upper quasi-isometry bound gives
$$
 |\Psi(a)|_{S_M},|\Psi(b)|_{S_M}
 \le(\Lambda C_R+A)r_j
$$
for $r_j\ge1$. Applying \eqref{eq:coordinate-pansu-convergence} in $N$ and $M$ gives
\begin{equation}\label{eq:cone-coarse-lip}
 \Lambda^{-1}d_{N,\infty}(x,y)-\omega_j(R)
 \le d_{M,\infty}(\Psi_j(x),\Psi_j(y))
 \le\Lambda d_{N,\infty}(x,y)+\omega_j(R),
\end{equation}
where $\omega_j(R)\to0$.

Choose a countable dense set $D=\{z_1,z_2,\ldots\}\subseteq N_\infty$ with $z_1=e$. Put $R_i=d_{N,\infty}(e,z_i)+1$. For $i\ge2$, choose $x_{i,j}\in D_j^N$ such that
$$
 d_{N,\infty}(x_{i,j},z_i)
 \le\theta_j^N(R_i)+\frac1j,
$$
and take $x_{1,j}=e$. For each fixed $i$, the points $x_{i,j}$ lie in a fixed $d_{N,\infty}$-ball, and \eqref{eq:cone-coarse-lip} shows the same for $\Psi_j(x_{i,j})$. Properness and a diagonal argument give a subsequence such that
$$
 \Psi_j(x_{i,j})\longrightarrow\Psi_\infty(z_i)
 \qquad\text{for every }i.
$$
Passing to the limit in \eqref{eq:cone-coarse-lip} gives
$$
 \Lambda^{-1}d_{N,\infty}(z_i,z_k)
 \le d_{M,\infty}(\Psi_\infty(z_i),\Psi_\infty(z_k))
 \le\Lambda d_{N,\infty}(z_i,z_k).
$$
Hence $\Psi_\infty$ extends uniquely to a $\Lambda$-bi-Lipschitz map $N_\infty\to M_\infty$ with $\Psi_\infty(e)=e$.

We next prove uniform convergence on balls. Fix $R>0$ and $0<\eta<1/2$. Choose a finite $\eta/2$-net $F=\{z'_{1},\dots,z'_{K}\}$ of $\overline {B_{N,\infty}(e,R)}$. For each $k=1,\dots,K$, choose $z_{i_k}\in D$ such that $d_{N,\infty}(z_{i_k},z'_k)<\eta/2$ and set $F:=\{z_{i_1},\dots,z_{i_K}\}\subseteq D$. For these finitely many points,
$$
 d_{N,\infty}(x_{i_k,j},z_{i_k})\to0,
 \qquad
 d_{M,\infty}(\Psi_j(x_{i_k,j}),\Psi_\infty(z_{i_k}))\to0
$$
uniformly in $k$. Given $x\in D_j^N\cap\overline {B_{N,\infty}(e,R)}$, choose $z_{i_k}\in F$ with $d_{N,\infty}(x,z_{i_k})\le\eta$. For large $j$, both $x$ and $x_{i_k,j}$ belong to $\overline {B_{N,\infty}(e,R+2)}$. Equation \eqref{eq:cone-coarse-lip} and the Lipschitz bound for $\Psi_\infty$ give for large enough $j$, uniformly in $x$,
$$
 d_{M,\infty}(\Psi_j(x),\Psi_\infty(x))
 \le(2\Lambda+1)\eta+\omega_j(R+2).
$$
Thus $\limsup_j\gamma_j(R)\le(2\Lambda+1)\eta$, and the arbitraryness of $\eta$ proves \eqref{eq:uniform-conv}.

It remains to prove surjectivity. Let $y\in M_\infty$ and put $R_y=d_{M,\infty}(e,y)+1$. Choose $\bar y_j\in D_j^M$ with
$$
 d_{M,\infty}(\bar y_j,y)
 \le\theta_j^M(R_y)+\frac1j,
$$
and write $\bar y_j=\sigma_{r_j}^M(w_j)$. The points $\bar y_j$ lie in a large enough $d_{M,\infty}$-ball for all $j$, so \eqref{eq:cone-bounded-lift} gives $|w_j|_{S_M}\le C_yr_j$. Since $\Psi$ is coarsely surjective, there exists $B>0$ such that for all $j$, there exists $u_j\in N$ such that $d_M(\Psi(u_j),w_j)\le B$. The lower quasi-isometry bound gives
$$
 |u_j|_{S_N}
 \le\Lambda(|w_j|_{S_M}+A+B)\le C_y'r_j\quad (C_y'=\Lambda(C_y+A+B)).
$$
Put $x_j=\sigma_{r_j}^N(u_j)$. Using \eqref{eq:coordinate-pansu-convergence}, we get that $(x_j)$ lies in a fixed $d_{N,\infty}$-ball and for $C'=B+C_y'$,
$$
 d_{M,\infty}(\Psi_j(x_j),\bar y_j)
 \le\frac{B}{r_j}
 +\sup_{\substack{a,b\in M\\|a|_{S_M},|b|_{S_M}\le C'r_j}}
 \left|
 \frac{d_M(a,b)}{r_j}
 -d_{M,\infty}(\sigma_{r_j}^M(a),\sigma_{r_j}^M(b))
 \right|,
$$
which tends to zero as $j\to\infty$. After passing to a further subsequence, compactness of closed $d_{N,\infty}$-balls gives $x_j\to x_\infty$. Choose $R_0$ so that $x_\infty$ and all $x_j$ lie in $B_{N,\infty}(e,R_0)$. Then \eqref{eq:uniform-conv} gives
$$
\begin{aligned}
 d_{M,\infty}(\Psi_\infty(x_\infty),y)
 &\le\Lambda d_{N,\infty}(x_\infty,x_j)
       +\gamma_j(R_0)
       +d_{M,\infty}(\Psi_j(x_j),y)
 \to0 \text{ as }j\to\infty.
\end{aligned}
$$
Thus $\Psi_\infty$ is surjective, and hence a bi-Lipschitz homeomorphism.
\end{proof}

\section*{Acknowledgements}
The first author was supported in part by SEED Grant RD/0519-IRCCSH0-024 and was a visitor at MPIM Bonn during the initial stages of this work. The second author was supported in part by the PMRF. The third author was supported in part by an IIT Bombay IRCC fellowship, a TIFR Mumbai postdoctoral fellowship, and an NBHM postdoctoral fellowship (No. 0204/17/2025/R\&D-II/12398). The authors are thankful to Roman Sauer for helpful discussions, and Lodha Mathematical Sciences Institute for organizing the programme on Analytic Group Theory. The authors are also grateful to IIT Bombay for providing ideal working conditions.

\bibliographystyle{alpha}

\bibliography{references1}

@article {Breuillard2014,
    AUTHOR = {Breuillard, Emmanuel},
     TITLE = {Geometry of locally compact groups of polynomial growth and
              shape of large balls},
   JOURNAL = {Groups Geom. Dyn.},
  FJOURNAL = {Groups, Geometry, and Dynamics},
    VOLUME = {8},
      YEAR = {2014},
    NUMBER = {3},
     PAGES = {669--732},
      ISSN = {1661-7207,1661-7215},
   MRCLASS = {22D05 (22E25 53C17)},
  MRNUMBER = {3267520},
MRREVIEWER = {Stefan\ Waldmann},
       DOI = {10.4171/GGD/244},
       URL = {https://doi.org/10.4171/GGD/244},
}

@article {BreuillardLeDonne2013,
    AUTHOR = {Breuillard, Emmanuel and Le Donne, Enrico},
     TITLE = {On the rate of convergence to the asymptotic cone for
              nilpotent groups and sub{F}insler geometry},
   JOURNAL = {Proc. Natl. Acad. Sci. USA},
  FJOURNAL = {Proceedings of the National Academy of Sciences of the United
              States of America},
    VOLUME = {110},
      YEAR = {2013},
    NUMBER = {48},
     PAGES = {19220--19226},
      ISSN = {0027-8424,1091-6490},
   MRCLASS = {20F65 (53C17 53C23 53C60)},
  MRNUMBER = {3153949},
MRREVIEWER = {Igor\ Belegradek},
       DOI = {10.1073/pnas.1203854109},
       URL = {https://doi.org/10.1073/pnas.1203854109},
}

@article {Sauer2006,
    AUTHOR = {Sauer, R.},
     TITLE = {Homological invariants and quasi-isometry},
   JOURNAL = {Geom. Funct. Anal.},
  FJOURNAL = {Geometric and Functional Analysis},
    VOLUME = {16},
      YEAR = {2006},
    NUMBER = {2},
     PAGES = {476--515},
      ISSN = {1016-443X,1420-8970},
   MRCLASS = {20F65 (20F16 20J06)},
  MRNUMBER = {2231471},
MRREVIEWER = {Wolfgang\ Pitsch},
       DOI = {10.1007/s00039-006-0562-y},
       URL = {https://doi.org/10.1007/s00039-006-0562-y},
}

@article {Pansu89,
    AUTHOR = {Pansu, Pierre},
     TITLE = {M\'etriques de {C}arnot-{C}arath\'eodory et quasiisom\'etries
              des espaces sym\'etriques de rang un},
   JOURNAL = {Ann. of Math. (2)},
  FJOURNAL = {Annals of Mathematics. Second Series},
    VOLUME = {129},
      YEAR = {1989},
    NUMBER = {1},
     PAGES = {1--60},
      ISSN = {0003-486X,1939-8980},
   MRCLASS = {53C20 (22E40)},
  MRNUMBER = {979599},
MRREVIEWER = {Gudlaugur\ Thorbergsson},
       DOI = {10.2307/1971484},
       URL = {https://doi.org/10.2307/1971484},
}

@article {Shalom2004,
    AUTHOR = {Shalom, Yehuda},
     TITLE = {Harmonic analysis, cohomology, and the large-scale geometry of
              amenable groups},
   JOURNAL = {Acta Math.},
  FJOURNAL = {Acta Mathematica},
    VOLUME = {192},
      YEAR = {2004},
    NUMBER = {2},
     PAGES = {119--185},
      ISSN = {0001-5962},
   MRCLASS = {20F65 (43A07)},
  MRNUMBER = {2096453},
MRREVIEWER = {Peter A. Linnell},
       DOI = {10.1007/BF02392739},
       URL = {https://doi.org/10.1007/BF02392739},
}

@article {Guivarch1973,
    AUTHOR = {Guivarc'h, Yves},
     TITLE = {Croissance polynomiale et p\'eriodes des fonctions
              harmoniques},
   JOURNAL = {Bull. Soc. Math. France},
  FJOURNAL = {Bulletin de la Soci\'et\'e{} Math\'ematique de France},
    VOLUME = {101},
      YEAR = {1973},
     PAGES = {333--379},
      ISSN = {0037-9484},
   MRCLASS = {22D05 (43A05 60B15)},
  MRNUMBER = {369608},
MRREVIEWER = {H.\ Heyer},
       URL = {http://www.numdam.org/item?id=BSMF_1973__101__333_0},
}

@book {DrutuKapovich2018,
    AUTHOR = {Dru\c tu, Cornelia and Kapovich, Michael},
     TITLE = {Geometric group theory},
    SERIES = {American Mathematical Society Colloquium Publications},
    VOLUME = {63},
      NOTE = {With an appendix by Bogdan Nica},
 PUBLISHER = {American Mathematical Society, Providence, RI},
      YEAR = {2018},
     PAGES = {xx+819},
      ISBN = {978-1-4704-1104-6},
   MRCLASS = {20F65 (20E08 20F05 20F16 20F18 20F67 20F69 57M07)},
  MRNUMBER = {3753580},
MRREVIEWER = {Igor\ Belegradek},
       DOI = {10.1090/coll/063},
       URL = {https://doi.org/10.1090/coll/063},
}

@article {AntonelliYoung25,
    AUTHOR = {Antonelli, Gioacchino and Young, Robert},
     TITLE = {Ergodic maps and the cohomology of nilpotent {L}ie groups},
   JOURNAL = {Math. Ann.},
  FJOURNAL = {Mathematische Annalen},
    VOLUME = {392},
      YEAR = {2025},
    NUMBER = {4},
     PAGES = {5491--5519},
      ISSN = {0025-5831,1432-1807},
   MRCLASS = {22E25 (37A05 37A25)},
  MRNUMBER = {4958509},
MRREVIEWER = {Cheng\ Zheng},
       DOI = {10.1007/s00208-025-03208-4},
       URL = {https://doi.org/10.1007/s00208-025-03208-4},
}

@article {GotfredsenKyed21,
    AUTHOR = {Gotfredsen, Thomas and Kyed, David},
     TITLE = {Cohomological induction and uniform measure equivalence},
   JOURNAL = {Fund. Math.},
  FJOURNAL = {Fundamenta Mathematicae},
    VOLUME = {255},
      YEAR = {2021},
    NUMBER = {1},
     PAGES = {19--49},
      ISSN = {0016-2736,1730-6329},
   MRCLASS = {20F65 (20J06 22D05 57M07 57T10)},
  MRNUMBER = {4289662},
MRREVIEWER = {Ryo\ Ohashi},
       DOI = {10.4064/fm778-2-2021},
       URL = {https://doi.org/10.4064/fm778-2-2021},
}

@article {Pansu83,
    AUTHOR = {Pansu, Pierre},
     TITLE = {Croissance des boules et des g\'eod\'esiques ferm\'ees dans les
              nilvari\'et\'es},
   JOURNAL = {Ergodic Theory Dynam. Systems},
  FJOURNAL = {Ergodic Theory and Dynamical Systems},
    VOLUME = {3},
    NUMBER = {3},
      YEAR = {1983},
     PAGES = {415--445},
      ISSN = {0143-3857},
  MRNUMBER = {0741395},
       DOI = {10.1017/S0143385700002054},
       URL = {https://doi.org/10.1017/S0143385700002054},
}

@article {Nomizu54,
    AUTHOR = {Nomizu, Katsumi},
     TITLE = {On the cohomology of compact homogeneous spaces of nilpotent
              {L}ie groups},
   JOURNAL = {Ann. of Math. (2)},
  FJOURNAL = {Annals of Mathematics. Second Series},
    VOLUME = {59},
      YEAR = {1954},
     PAGES = {531--538},
      ISSN = {0003-486X},
  MRNUMBER = {0064057},
       DOI = {10.2307/1969716},
       URL = {https://doi.org/10.2307/1969716},
}

@book {Raghunathan72,
    AUTHOR = {Raghunathan, M. S.},
     TITLE = {Discrete subgroups of {L}ie groups},
    SERIES = {Ergebnisse der Mathematik und ihrer Grenzgebiete, Band 68},
 PUBLISHER = {Springer-Verlag, New York-Heidelberg},
      YEAR = {1972},
     PAGES = {ix+227},
  MRNUMBER = {0507234},
       DOI = {10.1007/978-3-642-86426-1},
       URL = {https://doi.org/10.1007/978-3-642-86426-1},
}

@article {LafontSchmidtVanLimbeek17,
    AUTHOR = {Lafont, Jean-Fran\c cois and Schmidt, Benjamin and van
              Limbeek, Wouter},
     TITLE = {Quasicircle boundaries and exotic almost-isometries},
   JOURNAL = {Ann. Inst. Fourier (Grenoble)},
  FJOURNAL = {Universit\'e{} de Grenoble. Annales de l'Institut Fourier},
    VOLUME = {67},
      YEAR = {2017},
    NUMBER = {2},
     PAGES = {863--877},
      ISSN = {0373-0956,1777-5310},
   MRCLASS = {53C23 (20F67 51F99)},
  MRNUMBER = {3669513},
MRREVIEWER = {Alain\ Valette},
       DOI = {10.5802/aif.3098},
       URL = {https://doi.org/10.5802/aif.3098},
}

\end{document}